\documentclass[reqno,11pt]{amsart}

\usepackage{diagrams}

\usepackage{amssymb,
			cancel,
			enumitem, 
			caption, 
			array, 
			tikz 
}

\usepackage{geometry}
\usepackage{mathrsfs}

\DeclareCaptionFormat{myfigformat}{#1#2#3\hrulefill}
\DeclareCaptionFormat{mytabformat}{~ \\[10pt] #1#2#3\hrulefill}
\usepackage[
	bookmarksnumbered,
	pdftitle={Interpolate acyclic CoHA}
	pdfauthor={Justin Allman}
	]
	{hyperref}
	
\usepackage{multirow}
	
\makeatletter
\newcommand{\leqnomode}{\tagsleft@true}
\newcommand{\reqnomode}{\tagsleft@false}
\makeatother

\newcommand{\curly}{\mathcal}
\newcommand{\G}{\curly{G}}
\newcommand{\Rep}{\boldsymbol{\mathrm{Rep}}}
\DeclareMathOperator{\Hom}{Hom}
\DeclareMathOperator{\codim}{codim}
\newcommand{\GL}{\mathrm{GL}}
\newcommand{\GGLL}{\boldsymbol{\GL}}

\newcommand{\kp}{\vdash}

\newcommand{\cP}{\curly{P}}

\newcommand{\cQ}{\curly{Q}^\bullet}
\newcommand{\coha}{\mathscr{H}}
\newcommand{\coho}{\mathrm{H}}
\newcommand{\F}{\curly{F}}
\newcommand{\E}{\curly{E}}
\newcommand{\cA}{\mathscr{A}}
\newcommand{\Fl}{\mathrm{Fl}}
\newcommand{\FFll}{\boldsymbol{\Fl}}
\newcommand{\Sym}{\mathfrak{S}}
\newcommand{\kpm}{{\boldsymbol{\mathrm{m}}}}

\newcommand{\QQ}{\mathbb{Q}}
\newcommand{\ZZ}{\mathbb{Z}}
\newcommand{\CC}{\mathbb{C}}
\newcommand{\EE}{\mathbb{E}}
\newcommand{\NN}{\mathbb{N}}
\newcommand{\A}{\mathbb{A}}

\newcommand{\union}{\cup}
\newcommand{\Union}{\bigcup}
\newcommand{\intersect}{\cap}

\newcommand{\dirsum}{\oplus}
\newcommand{\Dirsum}{\bigoplus}
\newcommand{\tensor}{\otimes}
\newcommand{\Tensor}{\bigotimes}
\newcommand{\compose}{\circ}
\newcommand{\includes}{\hookrightarrow}

\newcommand{\iso}{\cong}
\newcommand{\homeo}{\approx}
\newcommand{\hmtpc}{\simeq}

\newcommand{\introthm}[2]{%
	\vspace{0.5\baselineskip}%
	\noindent\textbf{Theorem~#1}.~\textit{#2}%
	\vspace{0.5\baselineskip}}
	
\newcommand{\introcor}[2]{%
	\vspace{0.5\baselineskip}%
	\noindent\textbf{Corollary~#1}.~\textit{#2}%
	\vspace{0.5\baselineskip}}

\theoremstyle{plain}
\newtheorem{prop}{Proposition}[section]
\newtheorem{lem}[prop]{Lemma}

\newtheorem{cor}[prop]{Corollary}
\newtheorem{thm}[prop]{Theorem}
\newtheorem*{thm*}{Theorem}

\theoremstyle{definition}
\newtheorem{defn}[prop]{Definition}
\newtheorem{notn}[prop]{Notation}

\theoremstyle{remark}
\newtheorem{remark}[prop]{Remark}
\newtheorem{example}[prop]{Example}

\title[Structure for CoHA of acyclic quivers]{A family of structure isomorphisms for the Cohomological Hall Algebra of an acyclic quiver}

\author[J.~Allman]{Justin Allman}
\address{Department of Mathematics \\ US Naval Academy \\ Annapolis, MD}
\email{allman@usna.edu}


\begin{document}

\begin{abstract}
For any acyclic quiver, we establish a family of structure isomorphisms for its cohomological Hall algebra (CoHA).  The family is parameterized by partitions of the quiver into Dynkin subquivers. For each such partition, we write the domain of our isomorphism as a tensor product of subalgebras in two ways. In the first, each tensor factor is isomorphic to the CoHA of the quiver with a single vertex and no arrows. In the second, the tensor factors are each isomorphic to the CoHAs of the corresponding Dynkin subquivers. When the quiver is already an orientation of a simply-laced Dynkin diagram, our results interpolate between isomorphisms proved by Rim\'anyi. Such CoHA decompositions appear in prior work of Davison--Meinhardt and Franzen--Reineke, but our proof gives an explicit topological realization of these results. As a consequence of our method, we deduce that certain structure constants in the CoHA naturally arise as CoHA products of classes of Dynkin quiver polynomials. 
\end{abstract}

\maketitle


\section{Introduction}
\label{s:intro}

To every quiver $Q$, one can associate its cohomological Hall algebra (CoHA) $\coha(Q) = \coha$, which Kontsevich and Soibelman defined in their seminal paper \cite{mkys2011}. The CoHA is inspired by physics; it is designed to model the algebra of BPS states in string theory. Mathematically speaking, let $\gamma$ denote a dimension vector for $Q$. As a vector space, the CoHA is a direct sum over all dimension vectors $\coha = \Dirsum_\gamma \coha_\gamma$. The summands are the equivariant cohomology algebras $\coha_\gamma = \coho^\bullet_{\GGLL_\gamma}(\Rep_\gamma)$ where $\Rep_\gamma$ is the space of $Q$-representations of dimension $\gamma$ and $\GGLL_\gamma$ is the algebraic group which acts naturally on $\Rep_\gamma$ by simultaneously changing basis at each vertex. The novel multiplication (generally noncommutative) endowed by Kontsevich--Soibelman on the CoHA, and which we denote by $*$ throughout the paper, respects the dimension vector grading and encodes important combinatorial, algebraic, and geometric aspects of the quiver and its representations, see \cite{mkys2011,ae2012,rr2013,bd2017,hfmr2018}. 

The structure of $\coha$ (and its modules) has been the subject of much study \cite{mkys2011,ae2012,rr2013,zc2014,bd2017,hf2016,hf2018,hfmr2018,hfmr2019}. In this paper, we restrict to the case of quivers $Q$ which are acyclic, i.e., have no oriented cycles. Given a quiver, one problem is to determine subalgebras over which $\coha$ is generated. Our main results provide one possible point of view on the CoHA isomorphism problem; it is our hope that readers familiar with characteristic classes of singular varieties (specifically the story of quiver polynomials) will find our description of the CoHA approachable. For example, we state one of our main results below. 

\introcor{\ref{cor:dynkin.decomp}}{
For certain choices of Dynkin subquivers $Q^1,\ldots,Q^\ell$ of $Q$ (we call such a choice an \emph{admissible Dynkin subquiver partition}) the (left to right) $*$-multiplication induces an isomorphism \[\pushQED{\qed}\coha(Q^1) \tensor \cdots \tensor \coha(Q^\ell) \stackrel{\iso}{\longrightarrow} \coha(Q). \qedhere\popQED \]}

As in the statement above, the main results in this paper live in a family parameterized by a choice of an \emph{admissible Dynkin subquiver partition} of $Q$ \cite{ja2018}. The seminal work of Gabriel \cite{pg1972} established the \emph{Dynkin quivers} (orientations of simply-laced Dynkin diagrams of type $A$, $D$, or $E$) as the ``base case'' for considering the representations of acyclic quivers. An important chapter in the study of Dynkin quiver representations is the theory of \emph{quiver polynomials}. These are characteristic fundamental classes associated to the closure (in $\Rep_\gamma$) of a $\GGLL_\gamma$-orbit (equivalently, the locus of quiver representations of a fixed isomorphism type). Their remarkable positivity and stability properties have been explored by many authors, and have remained at the frontiers of algebraic combinatorics for two decades; see \cite{abwf1999,ab2002.qv,abakhtay2004,ab2005.alt,abakhtay2005,abfsay2005,em2005,akemms2006,ab2008,rr2014,ja2014.ir,rkakjr2019}. In fact, the quiver polynomials are distinguished elements of the CoHA since, given a quiver orbit $\Omega \subset \Rep_\gamma$, we have
\[
[\overline{\Omega}] \in \coho^\bullet_{\GGLL_\gamma}(\Rep_\gamma) = \coha_\gamma.
\]
For Dynkin quivers, the $\GGLL_\gamma$-orbits in $\Rep_\gamma$ correspond to so-called \emph{Kostant partitions} $\kpm \kp \gamma$ (see Section \ref{ss:dynkins}). Given an admissible Dynkin subquiver partition $\{Q^1,\ldots,Q^\ell\}$ for $Q$, let $\Omega_{\kpm^j}$ denote the Dynkin quiver orbit in $\Rep_{\gamma^j}(Q^j)$ corresponding to the Kostant partition $\kpm^j \kp \gamma^j$, where $\gamma^j$ is a dimension vector for $Q^j$, and so $\gamma = \sum_{j}\gamma^j$ is a dimension vector for $Q$. The following formula in the CoHA shows that the Dynkin quiver polynomials play a central role in the algebra structure of any acyclic CoHA. 

\introcor{\ref{cor:dynkin.orbit.mult}}{
In terms of the isomorphism of Corollary \ref{cor:dynkin.decomp}, we have
\[
	\left( \left[ \overline{\Omega_{\kpm^1}(Q^1)} \right] \in \coha_{\gamma^1}(Q^1)\right)
		* \cdots * 
		\left( \left[ \overline{\Omega_{\kpm^\ell}(Q^\ell)} \right] \in \coha_{\gamma^\ell}(Q^\ell) \right)
			= \left[ \overline{\eta}_\kpm \right] \in \coha_\gamma(Q)
	\]
where $\eta_\kpm \subset \Rep_\gamma(Q)$ is a distinguished subvariety we call a \emph{quiver stratum} (see Section \ref{s:imp.vars}).\qed
}

We now describe some historical advances for CoHA isomorphisms to put our results in context. The structure of $\coha$ for the quiver with one vertex and no loops, i.e., the Dynkin diagram of type $A_1$, was determined already in the original work of Kontsevich and Soibelman \cite[Section~2.5]{mkys2011}. $\coha(A_1)$ is an exterior algebra on countably many generators; we recount this example in Section \ref{ss:coha.A1}. Efimov proved an analogous result, conjectured by Kontsevich--Soibelman, for any symmetric quiver \cite{ae2012}. Efimov found primitive  generators for the CoHA of a symmetric quiver, and the description in this case established the positive integrality of the associated Donaldson--Thomas invariants. We highlight three other relevant points in this story (at times considering non-acyclic quivers).

I. Kontsevich and Soibelman also considered the quiver $1 \leftarrow 2$ (an orientation of the $A_2$ Dynkin diagram). They gave a decomposition of $\coha(A_2)$ in two ways: first as a tensor product of subalgebras parameterized by simple roots of the corresponding root system, and second as a tensor product of subalgebras parameterized by the positive roots \cite[Section~2.8]{mkys2011}. That is, they described subalgebras $\cA_{1,0}$, $\cA_{0,1}$, and $\cA_{1,1}$ such that each is isomorphic to $\coha(A_1)$ and furthermore
\begin{equation}
	\label{eqn:intro.A2} 
	\cA_{1,0} \tensor \cA_{0,1} 
	\stackrel{\iso}{\longrightarrow} \coha(A_2) 
	\stackrel{\iso}{\longleftarrow} \cA_{0,1} \tensor \cA_{1,1} \tensor \cA_{1,0}
\end{equation}
when the $*$-multiplication is carried out from left to right. Here, the passage to Poincar\'e series reproduces the quantum pentagon identity, going back to the work of Faddeev--Kashaev \cite{lfrk1994}.

II. Rim\'anyi \cite{rr2013} proved that the two decompositions from \eqref{eqn:intro.A2} generalized to the case when $Q$ is an orientation of any simply-laced Dynkin diagram (i.e., of type $A$, $D$, or $E$). In particular, there are subalgebras $\cA_\beta$ for each positive root $\beta$ (and hence for every simple root) such that
	\begin{equation}
		\label{eqn:intro.RR}
	\left(\Tensor^{\to}_{\text{$\alpha$ a simple root}} \cA_\alpha \right)
	\stackrel{\iso}{\longrightarrow} \coha
	\stackrel{\iso}{\longleftarrow} 
	\left(\Tensor^{\to}_{\text{$\beta$ a positive root}} \cA_\beta \right)
	\end{equation}
where the arrows over the tensor products indicate the $*$-multiplications must be carried out in a certain prescribed order. Each tensor factor subalgebra in \eqref{eqn:intro.RR} is isomorphic to $\coha(A_1)$.

III. More recently, works of Davison--Meinhardt \cite{bdsm2016} and Franzen--Reineke \cite{hfmr2018} give more refined tensor product decompositions which hold for any quiver. In the case of Davison--Meinhardt, the result generalizes completely to any so-called \emph{quiver with potential} (for acyclic quivers, any potential must necessarily be zero). Their CoHA isomorphisms are parameterized by a fixed stability condition $\theta$ and the tensor product is over the associated semi-stable CoHAs (with potential) taken in descending order of slope $\mu$ (see \cite[Theorem~6.1]{hfmr2018} and \cite[Theorem~D]{bdsm2016})
	\begin{equation}
		\label{eqn:intro.FR}
		\Tensor^{\to}_{\mu} \coha^{\theta\text{-sst},\mu} \,
		\stackrel{\iso}{\longrightarrow}
		\coha.
	\end{equation}
In the case of orientations of type $A$, $D$, or $E$ Dynkin diagrams, the isomorphisms \eqref{eqn:intro.FR} interpolate between Rim\'anyi's isomorphisms \eqref{eqn:intro.RR}, but in general not all semistable CoHAs $\coha^{\theta\text{-sst},\mu}$ are isomorphic to $\coha(A_1)$. The algebras $\coha^{\theta\text{-sst},\mu}$ do admit understood presentations even if explicit descriptions are difficult in particular cases \cite[Section~8]{hfmr2018}. For example, Franzen and Reineke considered the application of \eqref{eqn:intro.FR} to the so-called Kroneker quiver $1 \leftleftarrows 2$ \cite{hfmr2019}.

\vspace{\baselineskip}
Our main theorem, from which Corollaries \ref{cor:dynkin.decomp} and \ref{cor:dynkin.orbit.mult} follow, is the following decomposition result which, in the case when $Q$ is Dynkin, interpolates between the lefthand and righthand decompositions of \eqref{eqn:intro.RR}.

\introthm{\ref{thm:main}}{
Let $\{Q^1,\ldots,Q^\ell\}$ be an admissible Dynkin subquiver partition of the acyclic quiver $Q$. If $r_j$ denotes the number of positive roots for $Q^j$ and $r=\sum_j r_j$, there exist subalgebras $\cA_{1},\ldots,\cA_{r} \subset \coha(Q)$ such that each $\cA_{u}$ is isomorphic to $\coha(A_1)$, and such that the left to right $*$-multiplication induces an isomorphism
\[\pushQED{\qed} \cA_{1} \tensor \cdots \tensor \cA_{r} {\longrightarrow\,} \coha(Q). \qedhere \popQED\] 
}

In the context of the results of Franzen--Reineke and Davison--Meinhardt, we comment that the methods of our paper adopt a fundamentally \emph{topological viewpoint}, whereas the stability condition approach is a fundamentally \emph{representation theoretical viewpoint}. Each of our admissible Dynkin subquiver partitions corresponds to a stability condition from the representation theory viewpoint. In this sense, our Theorem \ref{thm:main} is implicit in the work of Franzen--Reineke and Davison--Meinhardt. 

Thus, the major accomplishment of our paper is an \emph{explicit topological realization, in the language of characteristic classes, of the isomorphism of Equation \eqref{eqn:intro.FR}}. It would be interesting to consider if there exists stability conditions for which each factor $\coha^{\theta-\mathrm{sst},\mu}$ is isomorphic to $\coha(A_1)$, but which can not be realized topologically by an admissible Dynkin subquiver partition.

The structure of the paper is as follows. In Section \ref{s:quivers} we lay out notations, definitions, and relevant preliminary results related to quivers and the equivariant cohomology algebras $\coha_\gamma$. In Section \ref{s:coha} we recall the definition of the CoHA and describe the $*$-multiplication explicitly. In Section \ref{s:imp.vars} we describe several important varieties which are parameterized by combinatorial data associated to the quiver, and prove several results about their structure and relevance to the present setting. In Section \ref{s:q.dilog.poincare} we recall results regarding quantum dilogarithm identities from \cite{ja2018} and describe their interpretation as Poincar\'e series for the CoHA. Our main results appear in Section \ref{s:structure} and we complete the proof of Theorem \ref{thm:main} in Section \ref{s:main.thm.pf}.

\subsection*{Acknowledgements} We thank Rich\'ard Rim\'anyi, who introduced us first to the topic of CoHAs and quantum dilogarithm identities in 2014 at Chapel Hill, NC. Further, we acknowledge support in Summer 2019 from an Office of Naval Research Junior NARC grant.

\section{Quivers}
\label{s:quivers}

\subsection{Preliminaries}
	\label{ss:quivers.prelims}
A quiver $Q = (Q_0,Q_1,h,t)$ is a directed graph where $Q_0$ is a finite set of vertices, $Q_1$ is a finite set of edges whose elements are called \emph{arrows}, and $h:Q_1 \to Q_0$ and $t:Q_1\to Q_0$ are maps respectively called \emph{head} and \emph{tail}. The maps $h$ and $t$ encode the orientation of the arrows. A \emph{dimension} vector $\gamma = (\gamma(i))_{i\in Q_0}$ is a list of non-negative integers, one for each quiver vertex. In the sequel, we let $D$ denote the monoid of all dimension vectors.\footnote{When understood from context, we will omit reference to $Q$ in our notations; e.g., although the monoid of dimension vectors depends on $Q_0$, we opt to write $D$ instead of, say, $D_Q$.} In the rest of the paper, we use the notation $[p]:=\{1,2,\ldots,p\}$ for any $p\in \NN$, and $e_i \in D$ is the \emph{simple dimension vector} with $1$ at vertex $i$ and zeros elsewhere. Thus, assuming our quiver has $n$ vertices, we identify $Q_0 = [n]$ and write $D \iso \Dirsum_{i=1}^n \ZZ_{\geq 0}\cdot e_i$. For $\gamma \in D$, we form the \emph{space of quiver representations}
\begin{equation}
	\label{eqn:defn.Rep.gamma}
	\Rep_\gamma(Q) = \Rep_\gamma := \Dirsum_{a\in Q_1} \Hom\left( \CC^{\gamma(ta)},\CC^{\gamma(ha)} \right)
\end{equation}
with action of the \emph{base change group} $\GGLL_\gamma := \prod_{i\in Q_0} \GL(\CC^{\gamma(i)})$ given by
\begin{equation}
	\label{eqn:defn.GGLL.action}
	(g_i)_{i\in Q_0} \cdot (\phi_a)_{a\in Q_1} = (g_{ha} \phi_a g_{ta}^{-1})_{a\in Q_1}.
\end{equation}
We let $\chi:D \times D \to \ZZ$ denote the bilinear Euler form for $Q$, given by the formula
\begin{equation}
	\label{eqn:defn.Euler.form}
	\chi(\gamma_1,\gamma_2) = \sum_{i\in Q_0} \gamma_1(i)\gamma_2(i) - \sum_{a\in Q_1} \gamma_1(ta)\gamma_2(ha) \nonumber
\end{equation}
for any $\gamma_1,\gamma_2 \in D$. We also consider the opposite antisymmetrization of $\chi$,
\begin{equation}
	\label{eqn:defn.lambda}
	\langle\gamma_1,\gamma_2\rangle := \chi(\gamma_2,\gamma_1) - \chi(\gamma_1,\gamma_2).\nonumber
\end{equation}
Observe that for any $i,i'\in Q_0$
\begin{equation}
	\label{eqn:lambda.simples}
	\langle e_i,e_{i'}\rangle = \#\{a\in Q_1: ta=i,\,ha=i'\} - \#\{a\in Q_1: ta=i',\,ha=i\},
\end{equation}
that is, $\langle e_i,e_{i'} \rangle$ equals the number of arrows $i\to i'$ minus the number of arrows $i'\to i$.

A \emph{path} in $Q$ is a concatenation of arrows $a_\ell\,\cdots\,a_2\,a_1$ such that $ha_j = ta_{j+1}$ for all $j\in[\ell-1]$. Such a path is an \emph{oriented cycle} if $ha_\ell = ta_1$. A quiver is called \emph{acyclic} if it contains no oriented cycles. In particular, acyclic quivers contain no loop arrows. Observe that the underlying non-directed graph associated to $Q$ may contain cycles, even if $Q$ is acyclic.

In the rest of the paper, we assume $Q$ is an acyclic quiver, and we further assume that the set of vertices $Q_0 = [n]$ is ordered so that for every arrow $a\in Q_1$, we have $ha < ta$. Such an ordering is always possible in an acyclic quiver, but is not unique in general; see e.g., \cite[Exercise~1.5.2]{hdjw2017}. Moreover, when $Q$ is acyclic note that \eqref{eqn:lambda.simples} exactly counts (with a sign, but no cancellation) the number of arrows between the vertices $i$ and $i'$. Therefore, our assumption on the ordering of $Q_0$ implies that $\langle e_i,e_{i'} \rangle \leq 0$  whenever $i<i'$.

\subsection{Dynkin quivers}
	\label{ss:dynkins}
A quiver $Q$ is a \emph{Dynkin quiver} if it is an orientation of a simply-laced Dynkin diagram (i.e., of type $A$, $D$, or $E$). We identify the \emph{simple roots} of the associated root system with the set of simple dimension vectors $e_i$ for each $i\in Q_0$. Hence each element, $\beta$, of the set of \emph{positive roots} $\Phi_+(Q) = \Phi_+$ has the form
\[
\beta = \sum_{i\in Q_0} d^i_{\beta} e_i
\]
for some non-negative integers $d^i_\beta$. In this way, we realize positive roots as dimension vectors. As is our convention, we will suppress $Q$ in the notation when it is clear from context.

We have already described an ordering on the vertices $Q_0$, and hence on the simple roots. Namely we require that $e_{ha}$ precedes $e_{ta}$ for every arrow $a\in Q_1$. We call this the \emph{``head before tail'' order} on simple roots. We describe an order on the positive roots $\Phi_+$ by forcing the condition that
\begin{equation}
	\label{eqn:pos.rts.order}
	\beta' \text{~precedes~} \beta'' \implies \langle \beta',\beta'' \rangle \geq 0.
\end{equation}
The above condition on positive roots defines a partial order, from which we adopt any linear extension as a total order. The choice is not unique, but always exists for Dynkin quivers (see \cite{mr2010,rr2013,jarr2018,ja2018} for more details and equivalent formulations). Because this order was first utilized in Reineke's work \cite{mr2010} to prove Donaldson--Thomas type identities (a connection we will exploit later), we call the resulting total order on $\Phi_+$ a \emph{Reineke order}; this terminology is also consistent with \cite{rr2018}.

\begin{example}
	\label{ex:Reineke.order.A3}
Consider the equioriented $A_3$ quiver $1 \leftarrow 2 \leftarrow 3$. Note that we have already chosen the names of the vertices so that the simple roots are in ``head before tail'' order $e_1\prec e_2 \prec e_3$. A Reineke order for the associated positive roots is given by
\begin{equation}
	\label{eqn:reineke.order.A3}
	\begin{aligned}
	\beta_1 & = e_3, & \beta_2 & = e_2+e_3, & \beta_3 & = e_2, \\
	\beta_4 & = e_1 + e_2 + e_3, & \beta_5 &= e_1 + e_2, & \beta_6 &= e_1,
	\end{aligned}
\end{equation}
where one can check that $u<v$ implies $\langle \beta_u,\beta_v \rangle \geq 0$. Observe we can obtain another Reineke order by interchanging $e_2$ and $e_1+e_2+e_3$ above, since $\langle e_2,e_1+e_2+e_3 \rangle = 0$.
\end{example}

Given a dimension vector $\gamma \in D$, a choice of non-negative integers $\kpm = (m_\beta)_{\beta\in\Phi_+}$ is called a \emph{Kostant partition} of $\gamma$ if \[\gamma = \sum_{\beta\in\Phi_+} m_\beta\,\beta. \] When this happens, we write $\kpm \kp \gamma$. Gabriel's theorem \cite{pg1972} that the orbits are in one-to-one correspondence with Kostant partitions. In the sequel, we will denote the orbit corresponding to $\kpm \kp \gamma$ by $\Omega_{\kpm}(Q) = \Omega_{\kpm} \subset \Rep_\gamma$. When $\Phi_+$ is in Reineke order $\beta_1 \prec \cdots \prec \beta_r$, we will abuse notation and write $m_u= m_{\beta_u}$.

\begin{example}
	\label{ex:Orbit.A3}
For $Q$ the equioriented $A_3$ quiver of Example \ref{ex:Reineke.order.A3}, take the Reineke order in \eqref{eqn:reineke.order.A3}. Consider the Kostant partition with $m_1 = m_5 = 0$, $m_2=2$, and all other $m_u = 1$. For type $A$ quivers, one can draw a \emph{lacing diagram} (originally due to \cite{saadf1980}), corresponding to the orbit $\Omega_\kpm$. In this case, we have that $\Omega_\kpm$ corresponds to the lacing diagram
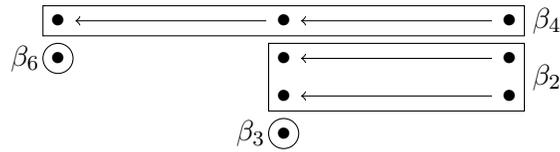
\begin{center}
	\begin{tikzpicture}
	\node (11) at (0,1.5) {$\bullet$};
	\node (12) at (0,1) {$\bullet$};
	\node (21) at (3,1.5) {$\bullet$};
	\node (22) at (3,1) {$\bullet$};
	\node (23) at (3,.5) {$\bullet$};
	\node (24) at (3,0) {$\bullet$};
	\node (31) at (6,1.5) {$\bullet$};
	\node (32) at (6,1) {$\bullet$};
	\node (33) at (6,.5) {$\bullet$};
	
	\draw[->] (31) -- (21);
	\draw[->] (32) -- (22);
	\draw[->] (33) -- (23);
	\draw[->] (21) -- (11);
	
	\node[draw, shape=rectangle, minimum width=6.4cm, minimum height=0.4cm, anchor=center] at (3,1.5) {};
	\node at (6.5,1.5) {$\beta_4$};
	
	\node[draw, shape=rectangle, minimum width=3.4cm, minimum height=0.9cm, anchor=center] at (4.5,0.75) {};
	\node at (6.5,.75) {$\beta_2$};
	
	\draw (0,1) circle [radius=0.2] node [left] {$\beta_6$\;};
	\draw (3,0) circle [radius=0.2] node [left] {$\beta_3$\;};

	\end{tikzpicture}
\end{center}
which says that $\Omega_\kpm$ consists of those quiver representations in $\Rep_{(2,4,3)}$ for which the lefthand mapping $\CC^4 \to \CC^2$ from vertex $2$ to vertex $1$ has rank $1$, the righthand mapping $\CC^3 \to \CC^4$ from vertex $3$ to vertex $2$ has full rank, and the intersection of the kernel of the lefthand map and the image of the righthand map has dimension $2$.
\end{example}

In the case of Dynkin quivers, the closure of the orbit $\overline{\Omega}_\kpm \subset \Rep_\gamma$ is called a \emph{quiver locus} and its equivariant fundamental class $[\overline{\Omega}_\kpm] \in \coho^\bullet_{\GGLL_\gamma}(\Rep_\gamma)$ is called a \emph{quiver polynomial}. These characteristic classes (and their K-theoretic analogues) have a rich history in their own right; they exhibit remarkable combinatorial and geometric properties and have connections to many other areas of algebraic combinatorics and algebraic geometry. The author suggests the seminal work of Buch--Fulton \cite{abwf1999}, as well as the more recent papers \cite{ab2008} and \cite{rkakjr2019} (and references therein), as a starting point on quiver polynomials. 

We observe that the quiver polynomial $[\overline{\Omega}_\kpm]$ is an element of $\coha_\gamma$, and hence an element of the CoHA. Rim\'anyi proved that the Dynkin quiver polynomials are akin to structure constants in Dynkin CoHAs \cite[Theorem~10.1]{rr2013}. In the sequel, we will see that these geometrically distinguished elements still play an interesting role in the algebraic structure of any acyclic CoHA.

\subsection{Equivariant cohomology algebras and their elements}
	\label{ss:quivers.eq.coho}
Fix the dimension vector $\gamma \in D$. In this paper, a central object we study is the equivariant cohomology algebra associated to the vector space \eqref{eqn:defn.Rep.gamma} and group action \eqref{eqn:defn.GGLL.action}; explicitly we set
\begin{equation}
	\label{eqn:defn.H.gamma}
	\coha_\gamma(Q) = \coha_\gamma := \coho^\bullet_{\GGLL_\gamma}(\Rep_\gamma)
\end{equation}
where here and throughout the paper we assume cohomology algebras have rational coefficients. As before, we omit $Q$ from the notation when it is clear from context. Now, for each quiver vertex $i\in Q_0$, let $\{\omega_{i,j}:j\in[\gamma(i)]\}$ denote a set of indeterminates (each with cohomological degree two), which we allow to represent the Chern roots of $\GL(\CC^{\gamma(i)})$. Since $\Rep_\gamma$ is a $\GGLL_\gamma$-equivariantly contractible vector space, we further identify
\begin{equation}
	\label{eqn:H.gamma.polynomials}
	\coha_\gamma \iso \coho^\bullet(B\GGLL_\gamma) 
		\iso \Tensor_{i\in Q_0} \QQ[\omega_{i,1},\ldots,\omega_{i,\gamma(i)}]^{\Sym_{\gamma(i)}}
\end{equation}
where $B$ denotes the {Borel construction} for equivariant cohomology and $\Sym_p$ is the symmetric group on $p$ letters. The action of each $\Sym_{\gamma(i)}$ is by permuting the variables $\{\omega_{i,j}:j\in[\gamma(i)]\}$. Hence, in the sequel we realize elements of $\coha_\gamma$ as polynomials $f(\omega_{i,j})$ which are separately symmetric in the variables $\omega_{i,j}$ for each $i\in Q_0$.

\subsection{Subquiver partitions}
	\label{ss:quivers.subquiver.partitions}
A quiver $Q$ is \emph{nonempty} if $Q_0 \neq \emptyset$, and $Q$ is \emph{connected} if its underlying non-oriented graph is connected. A \emph{subquiver} $Q'$ of $Q$ is a quiver with $Q'_0 \subseteq Q_0$ and $Q'_1 \subseteq Q_1$.

\begin{defn}
	\label{defn:dsqp}
Let $\cQ = \{Q^1,\ldots,Q^\ell\}$ be a set of subquivers of $Q$.
	\begin{enumerate}[label=(\alph*),leftmargin=*]
	\item We say that $\cQ$ is a \emph{subquiver partition} if each $Q^j$ is a non-empty connected subquiver, and $Q_0$ is the disjoint union of the $Q^j_0$ vertex sets.
	\item We say that $\cQ$ is a \emph{Dynkin subquiver partition} if it is a subquiver partition and each $Q^j$ is a Dynkin quiver.
	\item We say that $\cQ$ is an \emph{admissible subquiver partition} if it is a subquiver partition and the quiver obtained from $Q$ by contracting each subquiver $Q^j$ to a single vertex is also acyclic. 
	\item In the contracted quiver described by \ref{defn:dsqp}(c), note that the vertices are identified with the subquivers $Q^j \in \cQ$. We say that an admissible subquiver partition $\cQ$ is \emph{ordered} if $a$ is an arrow in the contracted quiver with $ha = Q^i$ and $ta = Q^j$, then $i<j$. 
	\end{enumerate}
\end{defn}

Definition \ref{defn:dsqp}(d) can be rephrased as saying that the subquivers are ordered so that they are in ``head before tail'' order in contracted quiver of \ref{defn:dsqp}(c). Hence every admissible subquiver partition can be ordered. Figure \ref{fig:adm.dsqp} provides an example. We further introduce the notation $\cQ_1 = \Union_{j\in[\ell]}Q^j_1$ to denote the set of arrows in $Q$ which appear in one of the $Q^j$ subquivers. Thus, the arrows of the contracted quiver described in Definition \ref{defn:dsqp}(c) are identified with $Q_1\setminus \cQ_1$.

\begin{figure}
\centering
\begin{tikzpicture}
	\node (00) at (0,0) {$\bullet$};
	\node (10) at (1,0) {$\bullet$};
	\node (20) at (2,0) {$\bullet$};
	\node (21) at (2,1) {$\bullet$};
	\node (30) at (3,0) {$\bullet$};
	\node (31) at (3,1) {$\bullet$};
	\node (40) at (4,0) {$\bullet$};
	\node (41) at (4,1) {$\bullet$};
	
	\draw [->] (10) -- (00);
	\draw [->] (10) -- (20);
	\draw [->] (10) -- (21);
	\draw [->] (21) -- (31);
	\draw [->] (20) -- (30);
	\draw [->] (31) -- (30);
	\draw [->] (40) -- (30);
	\draw [->] (40) -- (41);
	
	\draw[rounded corners, red] (-.25,0) -- (-.25,.25) -- (1,.25) -- (2,1.25) -- (2.25,.95) -- (1.5,.25) -- (2.25,.25) -- (2.25,-.25) -- (-.25,-.25) -- (-.25,0) ;
	\draw[rounded corners, red] (2.75, .5) -- (2.75, 1.25) -- (3.25, 1.25) -- (3.25,.25) -- (4.25,.25) -- (4.25, -.25) -- (2.75,-.25) -- (2.75, .5);
	\draw[red] (4,1) circle [radius=.25];
\end{tikzpicture}
	\caption{An admissible Dynkin subquiver partition which becomes ordered by the assignments: $Q^1 = A_1$, $Q^2$ an orientation of $A_3$, and $Q^3$ an orientation of $D_4$. Observe that the underlying non-oriented graph associated to $Q$ has a cycle, but that the quiver $Q$ is acyclic.}
	\label{fig:adm.dsqp}
\end{figure}
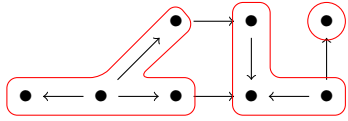

Given a dimension vector $\gamma$ for $Q$ we define the \emph{support} of $\gamma$ to be the set $\{i\in Q_0: \gamma(i) \neq 0\}$. If $Q'$ is a subquiver of $Q$ then we say \emph{$\gamma$ has its support in $Q'$} if the support of $\gamma$ is a subset of $Q'_0$. 

\begin{defn}
	\label{defn:consistent.gamma.lists}
Let $\cQ = \{Q^1,\ldots,Q^\ell\}$ be an admissible, ordered subquiver partition of $Q$. We say that the list of dimension vectors $\gamma_1,\ldots,\gamma_r$ is \emph{consistent with $\cQ$} if 
	\begin{itemize}[leftmargin=*]
	\item for each $u\in[r]$, there exists $j(u)\in[\ell]$ such that $\gamma_u$ has its support in $Q^{j(u)}$, and
	\item if $1\leq u<v \leq r$, then $1\leq j(u) \leq j(v) \leq \ell$.
	\end{itemize}
\end{defn}

\begin{example}
For the equioriented $A_3$ quiver $1 \leftarrow 2 \leftarrow 3$, we have the admissible, ordered, Dynkin subquiver partition $\cQ = \{Q^1,Q^2\}$ where $Q^1$ is the $A_1$-subquiver consisting of the singleton vertex $1$ and $Q^2$ is the $A_2$-subquiver $2\leftarrow 3$. Consider the following lists of dimension vectors for $Q$:
	\begin{enumerate}[label=(\roman*)]
	\item $\gamma_1 = (1,0,0)$, $\gamma_2 = (0,0,1)$, $\gamma_3 = (0,1,1)$, $\gamma_4 = (0,1,0)$;
	\item $\gamma_1 = (0,1,3)$, $\gamma_2 = (1,0,0)$, $\gamma_3 = (0,2,1)$, $\gamma_4 = (0,2,0)$;
	\item $\gamma_1 = (1,1,0)$, $\gamma_2 = (0,3,3)$, $\gamma_3 = (0,0,1)$.
	\end{enumerate}
Only (i) is consistent with $\cQ$. While (ii) satisfies the first bullet point of Definition \ref{defn:consistent.gamma.lists}, it fails the second with $u=1$ and $v=2$. The list (iii) fails the first bullet point since $\gamma_1 = (1,1,0)$ is not supported on either of the subquivers $Q^1$ or $Q^2$.
\end{example}

For the remainder of the subsection, we require that $\cQ$ is an admissible, ordered, Dynkin subquiver partition of $Q$. Hence each subquiver $Q^j$ has an associated set of positive roots $\Phi_+(Q^j)$ which are also naturally dimension vectors of $Q$ with entries of zero at each $i \notin Q^j_0$. Let $\Phi_+(\cQ) = \Union_{j\in[\ell]} \Phi_+(Q^j)$, and we extend the notion of a Reineke order on each $\Phi_+(Q^j)$ to the union $\Phi_+(\cQ)$ as follows. Write $r_j = \#\Phi_+(Q^j)$ and $\Phi_+(Q^j) = \{\beta^j_1,\ldots,\beta^j_{r_j}\}$ in a Reineke order for the Dynkin subquiver $Q^j$. Then a \emph{Reineke order} on $\Phi_+(\cQ)$ has the form
\[
(\beta^1_1 \prec \cdots \prec \beta^1_{r_1}) \prec (\beta^2_1 \prec \cdots \prec \beta^2_{r_2}) \prec \cdots \prec (\beta^\ell_{1} \prec \cdots \prec \beta^\ell_{r_\ell}).
\]
Observe that the ordering condition on $\cQ$ and the ordering condition \eqref{eqn:pos.rts.order} together imply that if $\beta'$ precedes $\beta''$ where $\beta',\beta''$ are 
	\begin{itemize}
	\item positive roots from \emph{the same} subquiver, then $\langle \beta',\beta''\rangle \geq 0$;
	\item positive roots from \emph{different} subquivers, then $\langle \beta',\beta''\rangle \leq 0$.
	\end{itemize}
From the above, we see that this Reineke order convention on $\Phi_+(\cQ)$ is equivalent to the ordering convention on the same set in \cite[Section~2.5]{ja2018}.

\section{Cohomological Hall algebras of quivers}
	\label{s:coha}

As a $D$-graded vector space, the \emph{cohomological Hall algebra} (aka CoHA) \emph{of a quiver $Q$} is
\begin{equation}
	\label{eqn:defn.CoHA}
	\coha(Q) = \coha := \Dirsum_{\gamma\in D} \coha_\gamma \nonumber
\end{equation}
where we recall that $\coha_\gamma$ is the equivariant cohomology algebra from \eqref{eqn:defn.H.gamma} and we think of the elements of $\coha_\gamma$ as polynomials via \eqref{eqn:H.gamma.polynomials}.

\subsection{Equivariant localization formula for CoHA multiplication}
	\label{ss:eq.local.coha.mult}
For $\gamma_1,\gamma_2\in D$ and $f_1\in \coha_{\gamma_1}$, $f_2\in\coha_{\gamma_2}$, we obtain $f_1 * f_2 \in \coha_{\gamma_1+\gamma_2}$ via the formula (see \cite[Theorem~2]{mkys2011})
\begin{multline}
		\label{eqn:coha.two.factor}
(f_1*f_2)(\omega_{1,1},\ldots,\omega_{1,\gamma_1(1)+\gamma_2(1)};\ldots;\omega_{n,1},\ldots,\omega_{n,\gamma_1(n)+\gamma_2(n)}) \\
 = \sum_{S_1 \in \binom{[\gamma_1(1) + \gamma_2(1)]}{\gamma_1(1)}} 
 \cdots 
	\sum_{S_n \in \binom{[\gamma_1(n) + \gamma_2(n)]}{\gamma_1(n)}}
		f_1(S_\bullet)f_2(\overline{S_\bullet}) \frac{\prod_{a\in Q_1} (\omega_{ha,\overline{S_{ha}}} - \omega_{ta,S_{ta}} )}{\prod_{i\in Q_0} (\omega_{i,\overline{S_i}} - \omega_{i,S_i} )}.
\end{multline}
To compute a term in the sum above, for each vertex $i\in Q_0$ we must choose a set \[S_i \in \binom{[\gamma_1(i) + \gamma_2(i)]}{\gamma_1(i)}\] which is a subset of $[\gamma_1(i) + \gamma_2(i)]$ with $\gamma_1(i)$ elements, and $\overline{S_i}$ denotes its complement in $[\gamma_1(i) + \gamma_2(i)]$ (and so has $\gamma_2(i)$ elements). Since $f_1 \in \coha_{\gamma_1}$ (respectively $f_2\in \coha_{\gamma_2}$) it is a polynomial separately symmetric in $n$ families of indeterminates, one family for each $i\in Q_0$, each with $\gamma_1(i)$ (resp.~$\gamma_2(i)$) variables. Hence, it makes sense to let $f_1(S_\bullet)$ denote the specialization $f_1(\omega_{i,j})$ for all $i\in Q_0$ and $j\in S_i$ (resp.~$f_2(\overline{S_\bullet})$ denotes the specialization $f_2(\omega_{i,j})$ for all $i\in Q_0$ and $j\in \overline{S_i}$). Finally, the products in the numerator and denominator should be interpreted according to the convention
\begin{equation}
	\label{eqn:prod.set.subscripts}
	\left(\omega_{x,U} - \omega_{y,V}\right) := \prod_{u\in U} \prod_{v\in V} \left(\omega_{x,u} - \omega_{y,v}\right).
\end{equation}
We comment that if $U$ or $V$ is empty, then the value of the product \eqref{eqn:prod.set.subscripts} is $1$. From the supersymmetry of $f_1$ and $f_2$, and the fact that we sum over all possible choices for $(S_1,\ldots,S_n)$, we see that the result of \eqref{eqn:coha.two.factor} will indeed have the supersymmetry required to be an element of $\coha_{\gamma_1+\gamma_2}$.

\begin{example}
	\label{ex:coha.mult}
Let $Q$ be the equioriented $A_3$ quiver $1 \leftarrow 2 \leftarrow 3$ as in Examples \ref{ex:Reineke.order.A3} and \ref{ex:Orbit.A3}. Consider $f(\omega_{1,1},\omega_{1,2}) \in \coha_{(2,0,0)}$ and $g(\omega_{2,1},\omega_{3,1}) \in \coha_{(0,1,1)}$. We have that
	\begin{align*}
	(f*g)(\omega_{1,1},\omega_{1,2},\omega_{2,1},\omega_{3,1}) 
		&= f\left(\omega _{1,1},\omega _{1,2}\right) g\left(\omega _{2,1},\omega _{3,1}\right)\\
	(g*f)(\omega_{1,1},\omega_{1,2},\omega_{2,1},\omega_{3,1}) 
		&=  f\left(\omega _{1,1},\omega _{1,2}\right) g\left(\omega _{2,1},\omega _{3,1}\right) \left(\omega _{1,1}-\omega _{2,1}\right) \left(\omega _{1,2}-\omega _{2,1}\right).
	\end{align*}
From this we observe that the product appears simpler when we multiply in an order which, when reading from left to right, goes \emph{against} the direction of an arrow. In this case, we mean that the dimension vector for $f$ is supported on the subquiver $Q^1$ consisting only of the vertex $1$, while the dimension vector for $g$ is supported on the subquiver $Q^2 = 2 \leftarrow 3$, and these two subquivers have the property that $Q^1$ sits at the head of an arrow from $Q^2$.
\end{example}

\begin{example}
	\label{ex:coha.mult.again}
Again we let $Q$ be the quiver $1 \leftarrow 2 \leftarrow 3$. This time, take $\omega_{2,1} \in \coha_{(0,1,0)}$ and $\omega_{3,1} \in \coha_{(0,0,1)}$ to see that
	\begin{align*}
		\omega_{2,1} * \omega_{3,1} &= \omega_{2,1}\omega_{3,1} \in \coha_{(0,1,1)} \\
		\omega_{3,1} * \omega_{2,1} &= \omega_{2,1}\omega_{3,1} (\omega_{2,1} - \omega_{3,1}) \in \coha_{(0,1,1)}
	\end{align*} 
where we again observe that going ``against'' the direction of an arrow (this time the arrow $2 \leftarrow 3$) makes the multiplication simpler. Now, consider $\omega_{3,1}$ not as an element of $\coha_{(0,0,1)}$, but instead as an element of $\coha_{(0,1,1)}$. This results in the products
	\begin{align*}
		\omega_{2,1} * \omega_{3,1} &= -\omega_{3,1} \in \coha_{(0,2,1)} \\
		\omega_{3,1} * \omega_{2,1} &= (\omega_{2,1}+\omega_{2,2})\,\omega_{3,1} - \omega_{3,1}^2 \in \coha_{(0,2,1)},
	\end{align*} 
illustrating the dependence on the dimension vector grading.
\end{example}

\subsection{Geometric definition for CoHA multiplication}
	\label{ss:geom.coha.mult}
The fact that \eqref{eqn:coha.two.factor} produces an honest polynomial (not a rational function) can be realized as a consequence of geometry. We will make only implicit use of the geometric definition of the CoHA multiplication in the sequel, but mention some relevant details here from \cite[Section~2.2]{mkys2011}. 

First, denote $\gamma = \gamma_1 + \gamma_2$. Set $P_{\gamma_1,\gamma_2}$ to be the parabolic subgroup of $\GGLL_{\gamma}$ which preserves, for each vertex $i\in Q_0$, the subspace $\CC^{\gamma_1(i)}\subseteq \CC^{\gamma(i)}$. That is, in the standard basis at each vertex, $P_{\gamma_1,\gamma_2}$ consists of upper block triangular matrices that, for each vertex factor, have diagonal blocks of sizes $\gamma_1(i)$ and $\gamma_2(i)$ for each $i\in Q_0$. Moreover, we let $\mathrm{Rep}_{\gamma_1,\gamma_2}$ denote the subspace of $\Rep_{\gamma}$ consisting of quiver representations which, for each arrow $a\in Q_1$, send the subspace $\CC^{\gamma_1(ta)} \subseteq \CC^{\gamma(ta)}$ into the subspace $\CC^{\gamma_1(ha)} \subseteq \CC^{\gamma(ha)}$. The multiplication mapping $\coha_{\gamma_1} \tensor \coha_{\gamma_2} \to \coha_{\gamma}$ is defined to be the result of the following compositions \cite[Section~2.2]{mkys2011}
\begin{multline}
	\label{eqn:geom.coha.mult}
\coho^\bullet_{\GGLL_{\gamma_1}}(\Rep_{\gamma_1})\tensor \coho^\bullet_{\GGLL_{\gamma_2}}(\Rep_{\gamma_2}) 
	\stackrel{\iso}{\longrightarrow} 
	\coho^\bullet_{\GGLL_{\gamma_1} \times \GGLL_{\gamma_2}}(\Rep_{\gamma_1}\dirsum\Rep_{\gamma_2}) \\
	\stackrel{\iso}{\longrightarrow} \coho^\bullet_{P_{\gamma_1,\gamma_2}}(\mathrm{Rep}_{\gamma_1,\gamma_2})
	\stackrel{\iota_*}{\longrightarrow} \coho^{\bullet+2c_1}_{P_{\gamma_1,\gamma_2}}(\Rep_{\gamma})
	\stackrel{p_*}{\longrightarrow} \coho^{\bullet+2c_1-2c_2}_{\GGLL_{\gamma}}(\Rep_{\gamma}). \nonumber
\end{multline} 
The first isomorphism is the K\"unneth formula and the second isomorphism follows from homotopy equivalence. The third map is the pushforward mapping along the closed equivariant embedding $\iota: \mathrm{Rep}_{\gamma_1,\gamma_2} \includes \Rep_\gamma$, and the fourth is the pushforward along the fibration $p: BP_{\gamma_1,\gamma_2} \to B\GGLL_\gamma$ with fiber $\GGLL_\gamma / P_{\gamma_1,\gamma_2}$ (using that $\Rep_\gamma$ is contractible). In particular, the fiber is a product of Grassmannians $\prod_{i\in Q_0} \mathrm{Gr}(\gamma_1(i),\CC^{\gamma(i)})$.

Thus, the formula \eqref{eqn:coha.two.factor} follows from the Atiyah--Bott, Berline--Vergne localization formula (\cite[(3.8)]{marb1984}, \cite{nbmv1982}) for equivariant pushforward mappings where $p_*$ is integration on a Grassmannian at each vertex. It is for this reason that we follow the terminology of \cite{rr2013} and call the formula \eqref{eqn:coha.two.factor} the \emph{equivariant localization formula} for the CoHA multiplication. The shifts in cohomological degree are 
\begin{align*}
	c_1 & = \dim_{\CC}(\Rep_{\gamma}) - \dim_{\CC}(\mathrm{Rep}_{\gamma_1,\gamma_2}) &
	c_2 & = \dim_{\CC}(\GGLL_\gamma/P_{\gamma_1,\gamma_2}).
\end{align*}
From this we see that while the $*$-multiplication respects the dimension vector grading of $\coha$, its relation to cohomological degree grading is
\[
	\coho^{b_1}_{\GGLL_{\gamma_1}}(\Rep_{\gamma_2}) 
	* \coho^{b_2}_{\GGLL_{\gamma_2}}(\Rep_{\gamma_2})
	\subseteq \coho^{b_1+b_2 - 2 \chi(\gamma_1,\gamma_2)}_{\GGLL_{\gamma}}(\Rep_{\gamma})
\]
where one checks that $c_2 = \sum_{i\in Q_0} \gamma_1(i)\gamma_2(i)$ and $c_1 = \sum_{a\in Q_1} \gamma_1(ta)\gamma_2(ha)$ so that $c_2-c_1 = \chi(\gamma_1,\gamma_2)$; e.g., compare this with the cohomological degree shifts in Examples \ref{ex:coha.mult} and \ref{ex:coha.mult.again}. 

\subsection{Multifactor products in the CoHA}
	\label{ss:multifactor.products}
	
A remarkable achievement of the seminal Kontsevich--Soibelman paper on CoHAs is the fact that the above multiplication is associative \cite[Theorem~1]{mkys2011}. Hence, the formula \eqref{eqn:coha.two.factor} has a well-defined multifactor version for products of the form \[ (f_1 \in \coha_{\gamma_1}) *\cdots*(f_r\in \coha_{\gamma_r}) \in \coha_{\sum_{u\in[r]} \gamma_u}.\] Given the required notational complexity, we do not write the general multifactor formula explicitly here. However we do mention that the rational functions which appear in the resulting formula match the equivariant localization formula for integration along multi-step flag manifolds (instead of simply Grassmannians), a connection we exploit in Section \ref{s:imp.vars}.

\subsection{The CoHA for $A_1$}
	\label{ss:coha.A1}
In this subsection we let $Q = A_1$, the quiver with a single vertex and no arrows. This example first appeared in Kontsevich--Soibelman's seminal work \cite[Section~2.5]{mkys2011}, and is replicated in introductory sections to many papers on CoHA. We reproduce it here for completion since the tensor product factors in our main theorem (Theorem \ref{thm:main}) are each subalgebras isomorphic to $\coha(A_1)$. 

For the quiver $A_1$, a dimension vector amounts to a choice of non-negative integer, and hence we have a decomposition $\coha(A_1) = \Dirsum_{r\geq 0} \coha_r$, where 
\begin{equation}
	\label{eqn:coha.r.A1}
\coha_r = \coho^\bullet(B\GL(\CC^r)) \iso \QQ[\omega_{1,1},\ldots,\omega_{1,r}]^{\Sym_r}.
\end{equation}
Throughout the remainder of the subsection, write $x_i = \omega_{1,i}$, where we warn that $x_i$ could refer to a variable in $\coha_n$ for any $n\geq i$. The location of $x_i$ will be clear from context or stated explicitly. We set $\psi_i = x_1^i \in \coha_1$. Now, let $\lambda = (\lambda_1,\ldots,\lambda_r)$ be a partition, i.e., with each $\lambda_u$ an integer such that $\lambda_1 \geq \cdots \geq \lambda_r\geq 0$. Applying the multi-factor multiplication gives
\begin{equation}
	\label{eqn:schur}
\psi_{\lambda_r} * \psi_{\lambda_{r-1} + 1} * \cdots * \psi_{\lambda_1+r-1} = s_\lambda(x_1,\ldots,x_r) \in \coha_r
\end{equation}
where $s_\lambda(x_1,\ldots,x_r)$ is the Schur symmetric function. Since the set \[\{s_\lambda(x_1,\ldots,x_r):\lambda\text{~a partition}\}\] is an additive basis for the ring of symmetric functions $\QQ[x_1,\ldots,x_r]^{\Sym_r} \iso \coha_r$ we see that $\coha(A_1)$ is generated by $\coha_1$. Moreover, a calculation with \eqref{eqn:coha.two.factor} shows that $\psi_i * \psi_j = -\psi_j * \psi_i$ for all $i$ and $j$. From this, one shows that $\coha(A_1)$ is the exterior algebra on countably many generators, namely the elements $\psi_i \in \coha_1$. That is, $\coha(A_1)  =  \bigwedge(\psi_0,\psi_1,\psi_2,\ldots)$.

\section{Important varieties from quiver data}
	\label{s:imp.vars}

The ideas underlying the definitions and results of this section go back at least to the work of Reineke \cite{mr2003} where desingularizations of quiver orbit closures were first described as incidence varieties in quiver flag varieties. These desingularizations appeared in \cite{rr2013} in the context of CoHAs for Dynkin quivers. Here, we extend the relevant results to the present general context of any acyclic quiver.

\subsection{Definitions}
	\label{ss:defn.imp.vars}

Let $\lambda = (\lambda_1,\ldots,\lambda_r)$ be a list of non-negative integers (not necessarily a partition) and let $|\lambda| = \sum_i \lambda_i$. We write $\Fl_\lambda$ to denote the flag variety
	\begin{equation}
		\label{eqn:flag.var}
		\Fl_\lambda = \left\{ 0 = E_0 \subseteq E_1 \subseteq \cdots \subseteq E_r = \CC^{|\lambda|}  \right\}
		\nonumber
	\end{equation}
where $E_u$ is a $\CC$-vector space and $\dim_\CC( E_u /E_{u-1}) = \lambda_u$ for each $u\in [r]$. Now, given a list of dimension vectors $\gamma_1,\ldots,\gamma_r$ for $Q$, we define the \emph{quiver flag variety}
	\begin{equation}
		\label{eqn:quiv.flag.var}
		\FFll_{\gamma_1,\ldots,\gamma_r} = \prod_{i\in Q_0} \Fl_{\gamma_1(i),\ldots,\gamma_r(i)}.
		\nonumber
	\end{equation}
That is, a point in $\FFll_{\gamma_1,\ldots,\gamma_r}$ is a list of $\CC$-vector spaces $(E_{i,u})$ with $i \in Q_0$, $u\in[r]$, and for each $i$ we have $E_{i,1} \subseteq \cdots \subseteq E_{i,r}$ (we assume $E_{i,0} = 0$) and such that $\dim_\CC(E_{i,u}/E_{i,u-1}) = \gamma_u(i)$. 

For each $i\in Q_0$, there are tautological bundles $\E_{i,u}$ over $\Fl_{\gamma_1(i),\ldots,\gamma_r(i)}$ whose fiber over the point $(E_{i,1}\subseteq \cdots \subseteq E_{i,r})$ is $E_{i,u}$. Moreover, for each $i$ and $u$, we obtain the rank $\gamma_i(u)$ quotient bundles $\F_{i,u} = \E_{i,u}/\E_{i,u-1}$. We also denote the pullbacks of these bundles (along the projections to each factor) to $\FFll_{\gamma_1,\ldots,\gamma_r}$ by $\E_{i,u}$ and $\F_{i,u}$. 

Given a subquiver partition $\cQ = \{Q^1,\ldots,Q^\ell\}$ for $Q$, we define the bundle on $\FFll_{\gamma_1,\ldots,\gamma_r}$
\begin{equation}
	\label{eqn:Gbundle.defn}
	\G(\cQ) = \Dirsum_{a\in \cQ_1} \Dirsum_{1\leq u<v \leq r} \Hom(\F_{ta,u}, \F_{ha,v}).
	\nonumber
\end{equation}
In the sequel, we denote the equivariant Euler class of the above bundle by $e(\G(\cQ))$.

For the dimension vector $\gamma = \sum_{u=1}^r \gamma_u$, we let 
\begin{equation}
	\label{eqn:pi.defn}
	\pi : \FFll_{\gamma_1,\ldots,\gamma_r} \times \Rep_\gamma \to \Rep_\gamma
	\nonumber
\end{equation}
denote the projection to the second factor. Moreover, we see that $\GGLL_\gamma$ also acts naturally on $\FFll_{\gamma_1,\ldots,\gamma_r}$ by the standard action of $\GL(\CC^{\gamma(i)})$ on $\CC^{\gamma(i)}$ at each vertex. Hence, we have a pushforward mapping $\pi_{*}$ in equivariant cohomology, whose target is $\coha_\gamma$.

\begin{defn}
	\label{defn:consistency.subset}
Given $\cQ = \{Q^1,\ldots,Q^\ell\}$ an admissible, ordered, subquiver partition, and the consistent list of dimension vectors $\gamma_1,\ldots,\gamma_r$ (with $\gamma = \sum_u \gamma_u$), define the associated \emph{consistency subset} of $\FFll_{\gamma_1,\ldots,\gamma_r} \times \Rep_\gamma$ to be
\begin{equation}
	\label{eqn:defn.incidence}
	\Sigma_{\gamma_1,\ldots,\gamma_r}(\cQ) 
		:= \left\{ 
			\left( (E_{i,u}), (\phi_a) \right) : \begin{array}{c}{\forall u\in[r]\text{~and~}\forall a\in \cQ_1} \\ {\text{we have~}\phi_a(E_{ta,u})\subseteq E_{ha,u}} \end{array}
			\right\}.
	\nonumber
\end{equation}
\end{defn}

We comment on the similarity of this incidence variety to that of \cite[Section~2]{mr2003} and \cite[Section~8]{rr2013}. In \cite{rr2013}, the variety above appears in the special case of $Q$ a Dynkin quiver and $\cQ = \{Q\}$. In that scenario, the seminal work of \cite{mr2003} provides an algorithm so that for appropriate choices of the dimension vectors $\gamma_1,\ldots,\gamma_r$, the consistency subset is a desingularization for a Dynkin quiver orbit closure. That is, the map $\pi$ gives a resolution of the singularities for a Dynkin quiver orbit when restricted to the consistency subset (we will utilize this result in our proofs of Propositions \ref{prop:consistency.subset.fund.class} and \ref{prop:consist.subset.desing}).

When $\cQ$ is a Dynkin subquiver partition, write $\Phi_+(\cQ) = \{\beta_1,\ldots,\beta_r\}$. A list $\kpm =(m_u)_{u\in[r]}$ is a \emph{$\cQ$-partition of $\gamma$} if $\sum_u m_u\,\beta_u = \gamma$. This generalizes the notion of Kostant partition (see Section \ref{ss:dynkins}) since if $Q$ is Dynkin and $\cQ = \{Q\}$, then $\kpm \kp \gamma$ is exactly the condition that $\kpm$ is a $\cQ$-partition. We also write $\kpm \kp \gamma$ when $\kpm$ is $\cQ$-partition of $\gamma$. Further, analogous to our abuse of notation in Section \ref{ss:dynkins}, we write \[m_u = m^j_{\beta^j_k}\] since for each $u\in[r]$ we have that $\beta_u \in \Phi_+(\cQ)$ is uniquely identified with some $\beta^j_k \in \Phi_+(Q^j)$ with $j\in[\ell]$ and $k\in[r_j]$. We further observe that $\kpm^j := (m^j_b)_{b\in\Phi_+(Q^j)}$ is a Kostant partition of the dimension vector $\gamma$ restricted to the vertices of $Q^j$. Let $\gamma^j$ denote the resulting dimension vector for $Q^j$. Hence, for each $j$, we see that a $\cQ$-partition $\kpm \kp \gamma$ amounts to a choice of a Dynkin quiver orbits $\Omega_{\kpm^j}(Q^j) \subset \Rep_{\gamma^j}(Q^j)$.

\begin{defn}
	\label{defn:m.stratum}
Suppose $\cQ$ is a Dynkin subquiver partition and $\kpm \kp \gamma$ is a $\cQ$-partition. The \emph{quiver stratum associated to $\kpm$} is the subspace
\begin{equation}
\label{eqn:quiver.strata.defn}
\eta_\kpm = \left\{ (\phi_a)_{a\in Q_1} \in \Rep_\gamma(Q) : (\phi_a)_{a\in Q^j_1} \in \Omega_{\kpm^j}(Q^j) \text{ for all $j\in[\ell]$} \right\} .
\nonumber
\end{equation}
That is, $\eta_\kpm$ consists of those quiver representations of $Q$ which, when restricted to the subquiver $Q^j$, lie in a specified Dynkin quiver orbit for all $j\in[\ell]$.
\end{defn}

\begin{notn}[Polynomials evaluated on bundles] 
	\label{notn:evaluate.bundles}
Suppose that we have a polynomial in many variables $f(x_{ij})$ where $i\in[n]$ and $j\in[r(i)]$ (for some natural numbers $n$ and $r(i)$) which is separately symmetric, for each $i$, in the set of variables $x_{i,\bullet}=\{x_{ij}:j\in[r(i)]\}$. Then, given $G$-equivariant vector bundles $\curly{V}_1,\ldots,\curly{V}_n$ on a $G$-space $X$ with the property that $\mathrm{rank}(\curly{V}_i) = r(i)$, we let $f(\curly{V}_\bullet) \in \coho^\bullet_G(X)$ represent the class obtained by evaluating $f$ on the Chern roots of the bundles $\curly{V}_1,\ldots,\curly{V}_n$.
\end{notn}

\subsection{Results for the important varieties}
	\label{ss:results.imp.vars}

We again turn our attention to the varieties, bundles, and mappings defined in Section \ref{ss:defn.imp.vars}.

\begin{prop}
	\label{prop:multifactor.pushforward.mult}
Let $\gamma_1,\ldots,\gamma_r$ be a list of dimension vectors which is consistent with the admissible, ordered,  subquiver partition $\cQ$ (the Dynkin condition is not required here). Then
\begin{equation}
	\label{eqn:multi.mult.as.pi}
	(f_1\in \coha_{\gamma_1}) * \cdots * (f_r\in\coha_{\gamma_r}) = \pi_{*}\left( e(\G({\cQ})) \cdot \prod_{u=1}^r f_u(\F_{\bullet,u})  \right).
\end{equation}
\end{prop}

\begin{proof}
The formula of \cite[Lemma~8.1]{rr2013} applies to the multiplication in our $\coha$, even though our quiver $Q$ is not assumed to be Dynkin. That is, we have from the equivariant localization description of the multifactor multiplication that
	\begin{equation}
		\label{eqn:rr.multi.lemma.int}
	f_1* \cdots *f_r = \int_{\FFll_{\gamma_1,\ldots,\gamma_r}} e(\G) \cdot \prod_{u=1}^r f_u(\F_{\bullet,u})
	\end{equation}
where $\G$ is the bundle
	\[
	\G = \Dirsum_{a \in Q_1} \Dirsum_{1\leq u<v \leq r} \Hom(\F_{ta,u}, \F_{ha,v})
	\]
and
\begin{equation}
	\label{eqn:rr.multi.lemma.int2}
	\int_{\FFll_{\gamma_1,\ldots,\gamma_r}}\,:\,
		\coho^\bullet_{\GGLL_\gamma}(\FFll_{\gamma_1,\ldots,\gamma_r})
		 \to \coho^\bullet_{\GGLL_\gamma}(pt)
\end{equation}
is the pushforward mapping in equivariant cohomology to a point. Observe that this $\G$ differs from $\G(\cQ)$ only in that the Whitney sum in $\G(\cQ)$ is over (possibly) fewer arrows. To justify the removal of these arrows from the sum, we prove Lemma \ref{lem:need.arrow.subset}; cf.~Example \ref{ex:coha.mult}.

In fact, we comment that one key point of our definitions up to now is that our ordering and admissibility criteria are chosen specifically to guarantee the truth of this lemma. In the statement below, we need the following definition. Given subquivers $Q'$ and $Q''$ of $Q$, we can form the subquiver $Q'\union Q''$ by taking the respective unions of their vertex and arrow sets. 

\begin{lem}
	\label{lem:need.arrow.subset}
Suppose $\cQ = \{Q^j:j\in[\ell]\}$ is an admissible, ordered, subquiver partition. Let $\gamma'$ and $\gamma''$ be dimension vectors with their respective supports in the subquivers $Q'=\Union_{j=1}^{j_1} Q^j$ and $Q''=\Union_{j=j_2}^\ell Q^{j}$ with $j_1\leq j_2$. Consider the product formula for $g_1 * g_2$ where $g_1\in \coha_{\gamma'}$ and $g_2\in\coha_{\gamma''}$. For any $a \notin \cQ_1$, we have that the numerator factor \[\left(\omega_{ha,\overline{S_{ha}}} -\omega_{ta,S_{ta}}\right) = 1.\]
\end{lem}

\begin{proof}[Proof of Lemma \ref{lem:need.arrow.subset}]
Take $a \notin \cQ_1$. In each case below, we argue that at least one of $S_{ta}$ or $\overline{S_{ha}}$ is empty. We consider the cases 

\begin{center} (I) $j_1 < j_2$ and (II) $j_1 = j_2$. \end{center} 

In case (I), the ordered condition on $\cQ$ implies that we can not have both $ha \in Q''_0$ and $ta \in Q'_0$. If $ha \notin Q''_0$, then the support hypothesis implies $\gamma''(ha) = 0$. Thus the only choice for $S_{ha}$ is the set $[\gamma'(ha)]$, and so $\overline{S_{ha}}=\emptyset$. Similarly, if $ta\notin Q'$, we get $\gamma'(ta) = 0$ and thus $S_{ta} = \emptyset$.

In case (II), write $j = j_1 = j_2$. The admissibility condition on $\cQ$ implies that we can not have both $ha \in Q^j_0$ and $ta\in Q^j_0$, otherwise we would have a loop in the contracted quiver from Definition \ref{defn:dsqp}(c). Suppose $ha \notin Q^j_0$ (the case $ta \notin Q^j_0$ is similar), and we consider two cases: $ha\in Q'_0$ or $ha\in Q''_0$. If $ha\in Q'_0$, then $ha \in Q^{j'}_0$ for some $j'<j$, and in particular the support hypothesis implies that $\gamma''(ha) = 0$; i.e., $\overline{S_{ha}} = \emptyset$. If $ha\in Q''_0$, then the ordering condition implies that $ta\in Q^{j''}_0$ with $j''>j$. Then the support condition implies that $\gamma'(ta) = 0$ from whence it follows that $S_{ta} = \emptyset$.
\end{proof}

\noindent \textit{Proof of Proposition \ref{prop:multifactor.pushforward.mult} continued}: Using the associativity of the $*$-multiplication, the consistency restriction on $\gamma_1,\ldots,\gamma_r$ implies that any two-factor product we encounter in the computation of $f_1 * \cdots * f_r$ will satisfy the hypotheses of Lemma \ref{lem:need.arrow.subset} for some $j_1$ and $j_2$. Hence by repeated application of the lemma we see that $e(\G)$ in \eqref{eqn:rr.multi.lemma.int} can be replaced with $e(\G(\cQ))$. The theorem then follows from the observation that $\Rep_\gamma$ is $\GGLL_\gamma$-equivariantly contractible, which identifies the integral map \eqref{eqn:rr.multi.lemma.int2} with the $\pi_{*}$ map.
\end{proof}

\begin{remark}
It is an illustrative exercise in the definitions up to this point in the paper to check that the polynomials $f_u$ are being evaluated on the appropriate number of variables on the righthand side of Equation \eqref{eqn:multi.mult.as.pi}.
\end{remark}

We prefer the formula \eqref{eqn:multi.mult.as.pi}, as opposed to the integral formula \eqref{eqn:rr.multi.lemma.int} in the proof of Proposition \ref{prop:multifactor.pushforward.mult}, because we can identify the Euler class of $\G(\cQ)$ with the (equivariant) fundamental class of the consistency subset.

\begin{prop}
	\label{prop:consistency.subset.fund.class}
$[\Sigma_{\gamma_1,\ldots,\gamma_r}(\cQ)] = e(\G(\cQ))$.
\end{prop}

\begin{proof}
We expand on the proof of \cite[Lemma~8.3]{rr2013}. Consider a torus fixed point $t \in \Sigma_{\gamma_1,\ldots,\gamma_r}(\cQ)$. On some neighborhood of $t$, say $U_t \subset \FFll_{\gamma_1,\ldots,\gamma_r} \times \Rep_\gamma$, we can (locally) choose subbundles $\overline{\F}_{i,u}\subset \E_{i,u}$ for all $i\in Q_0$ and $u\in[r]$, such that $\overline{\F}_{i,u} \dirsum \E_{i,u-1} = \E_{i,u}$. In particular, this means that on $U_t$ we have $\E_{i,u} = \Dirsum_{w\leq u} \overline{\F}_{i,w}$ for all $i$ and $u$.

We can form a vector bundle on $U_t$ (which is locally identified with $\G(\cQ)$) by
\[
\overline{\G} = \Dirsum_{a\in \cQ_1} \Dirsum_{u<v} \Hom\left(\overline{\F}_{ta,u},\F_{ha,v}\right).
\]
Recall that the Euler class corresponds to the fundamental class for a vanishing locus of a generic section (i.e., transverse to the zero section). Let $F_{i,u}$ denote the fiber of $\overline{\F}_{i,u}$ (it is a subspace of $E_{i,u}$). We have a natural section $\sigma:U_t \to \overline{\G}$ given by
\[
\left((E_{i,u})_{i\in Q_0,u\in[r]},(\phi_a)_{a\in Q_1}\right) 
	\longmapsto 
	\sum_{a\in\cQ_1,~u<v} \overline{\phi}_{a,u,v}
\]
with $\overline{\phi}_{a,u,v}$ defined to be the composition of the inclusion $F_{ta,u} \includes E_{ta,r}$ and the given linear map $\phi_a:E_{ta,r} \to E_{ha,r}$, followed by the quotient mapping to $E_{ha,r} / ( \Dirsum_{w \neq v} F_{ha,w} )$. The section $\sigma$ satisfies the genericity condition and, moreover, we now show that the zero locus of $\sigma$ is exactly $\Sigma_{\gamma_1,\ldots,\gamma_r}(\cQ) \intersect U_t$. 

Indeed, it immediately follows from the definitions that if $\left((E_{i,u}) , (\phi_a)\right) \in U_t$ has $\phi_a(E_{ta,u}) \subseteq E_{ha,u}$ for all $a \in \cQ_1$ and $u\in[r]$, then $\overline{\phi}_{a,u,v} = 0$ for all choices of $a,u,v$. Conversely, if $\overline{\phi}_{a,u,v}$ is identically zero for all $a,u,v$ we see that, fixing $a$ and $u$, we get that $\phi_a(F_{ta,u}) \subseteq \Dirsum_{w\neq v} F_{ha,w}$ for all $v>u$. Hence $\phi_a(F_{ta,u})\intersect F_{ha,v}=0$ for all $v>u$, and consequently we obtain $\phi_a(F_{ta,u}) \subseteq E_{ha,u}$. Since $u$ was arbitrary, we have \[\phi_a(E_{ta,u}) = \phi_a\left(\Dirsum_{w\leq u} F_{ta,w}\right) \subseteq \Dirsum_{w\leq u} F_{ha,w} = E_{ha,u}.\] Finally, since this holds for all choices of $a$, we have $\left((E_{i,u}) , (\phi_a)\right) \in U_t$ must be an element of $\Sigma_{\gamma_1,\ldots,\gamma_r}(\cQ)$.

Hence we have the string of equalities
\[
e(\G(\cQ))|_{t} = e(\overline{\G})|_t = [\Sigma_{\gamma_1,\ldots,\gamma_r}(\cQ)]|_t.
\]
Since $t$ is arbitrary, this holds at each torus fixed point, from whence the result follows from the localization theorem for torus equivariant cohomology.
\end{proof}

In the remainder of the section, let $\kpm$ be a $\cQ$-partition for the dimension vector $\gamma$. Observe that setting $\gamma_u = m_u \beta_u$ (with the roots $\beta_u$ in Reineke order) makes $\gamma_1,\ldots,\gamma_r$ a consistent list of dimension vectors with $\sum_u \gamma_u = \gamma$. Let $\pi_\kpm$ denote the restriction of $\pi$ to the consistency subset $\Sigma_{m_1\beta_1,\ldots,m_r\beta_r}(\cQ) \subseteq \FFll_{m_1\beta_1,\ldots,m_r\beta_r}\times \Rep_\gamma$.

\begin{prop}
	\label{prop:consist.subset.desing}
The mapping $\pi_\kpm: \Sigma_{m_1\beta_1,\ldots,m_r\beta_r}(\cQ) \to \Rep_\gamma$ has image $\overline{\eta}_\kpm$, and is a desingularization of the stratum closure $\overline{\eta}_\kpm$. As a consequence, when $\overline{\eta}_\kpm$ has rational singularities, we have \[\pi_*([\Sigma_{m_1\beta_1,\ldots,m_r\beta_r}(\cQ)]) = [\overline{\eta}_\kpm] \in \coha_\gamma.\] 
\end{prop}

\begin{proof}
When $Q$ is Dynkin and $\cQ = \{Q\}$, the fact that the consistency subset is a desingularization of the quiver orbit $\overline{\eta}_\kpm = \overline{\Omega}_\kpm$ (via the mapping $\pi$) was established by Reineke \cite{mr2003}. In the more general context stated above, $\eta_\kpm$ is a product of the Dynkin quiver orbits $\Omega_{\kpm^j}(Q^j)$ (one for each subquiver $Q^j \in \cQ$) and vector spaces $\Hom(\CC^{\gamma(ta)},\CC^{\gamma(ha)})$ with $a \in Q_1 \setminus \cQ_1$. Thus \[\overline{\eta}_\kpm  = \left(\prod_{j\in[\ell]} \overline{\Omega_{\kpm^j}(Q^j)}\right) \times \left(\prod_{a\notin \cQ_1} \Hom(\CC^{\gamma(ta)},\CC^{\gamma(ha)}) \right) \subseteq \Rep_\gamma.\]
The decomposition above means that the first assertion follows from repeated application of Reineke's desingularization theorem \cite[Theorem~2.2]{mr2003}.
\end{proof}

\begin{remark}[Regarding the second assertion of Proposition \ref{prop:consist.subset.desing}]
When $Q$ is a Dynkin quiver, Bobinski--Zwara have proven that $\overline{\Omega_\kpm}$ is guaranteed to have rational singularities provided $Q$ is an orientation of a type $A$ or $D$ Dynkin diagram \cite{gbgz2001,gbgz2002}. Therefore, when $\cQ$ consists only of type $A$ or $D$ subquivers, the Bobinski--Zwara results already guarantee that $\overline{\eta}_\kpm$ has rational singularities. To the author's knowledge, it is an open question whether type $E$ Dynkin quiver orbit closures are normal, Cohen--Macauley, and/or admit only rational singularities; we are aware of some partial affirmative results when all of the quiver's vertices are sinks/sources \cite{ks2015}. In any event, in the rest of the paper, we assume type $E$ orbit closures \emph{do} have only rational singularities, and so we apply Proposition \ref{prop:consist.subset.desing} (and also the next Proposition \ref{prop:eta.m.prod.of.1s}) freely in the sequel without further comment. We note, however, that our main theorem already prohibits orientations of $E_8$ as subquivers in $\cQ$, albeit for a different reason (a punctilio we discuss in Section \ref{s:structure}).
\end{remark}

\begin{prop}
	\label{prop:eta.m.prod.of.1s}
If $\overline{\eta}_\kpm \subset \Rep_\gamma$ has rational singularities, then its $\GGLL_\gamma$ equivariant fundamental class is given by
	\begin{equation}
		\label{eqn:eta.m.prod.of.1s}
		[\overline{\eta}_\kpm] = (1\in \coha_{m_1\beta_1}) * \cdots * (1 \in \coha_{m_r\beta_r}) \in \coha_\gamma.
		\nonumber
	\end{equation}
\end{prop}

\begin{proof}
Proposition \ref{prop:consist.subset.desing} says that $\pi_*([\Sigma_{m_1\beta_1,\ldots,m_r\beta_r}(\cQ)]) = [\overline{\eta}_\kpm]$. Hence, Proposition \ref{prop:consistency.subset.fund.class} then implies $[\overline{\eta}_\kpm] = \pi_*(e(\G(\cQ)))$ and so Proposition \ref{prop:multifactor.pushforward.mult} implies the result.
\end{proof}

\section{Quantum dilogarithms and Poincar\'e series of CoHAs}
	\label{s:q.dilog.poincare}

Let $q^{1/2}$ be an indeterminate (with square denoted $q$). Let $d$ be a positive integer, and set $\cP_d = \prod_{k=1}^{d}(1-q^k)^{-1}$. Furthermore, set $\cP_0 = 1$. We note that $\cP_d$ is the Poincar\'e series of the algebra $\coho^\bullet_{\GL(\CC^d)}(pt) = \coho^\bullet(B\GL(\CC^d))$. Given an indeterminate $z$ we define the \emph{quantum dilogarithm series}, an element of the algebra $\QQ(q^{1/2})[[z]]$, to be
	\begin{equation}
		\label{eqn:q.dilog.defn}
	\EE(z) = \sum_{d = 0}^\infty (-1)^d z^d\,q^{d^2/2}\,\cP_d.
	\end{equation}
If we allow $z$ to keep track of the dimension vector grading, and allow $q$ to keep track of the cohomological degree grading, we therefore notice that $\EE(z)$ is a $q$-shifted (by $q^{r^2/2}$) and twisted (by minus signs) Poincar\'e series for the CoHA $\coha(A_1)$. Indeed, setting $(\coha_r)_k$ to be the degree $2k$ part of $\coha_r$, i.e., $(\coha_r)_k \iso \coho^{2k}(B\GL(\CC^r))$, (refer to the notations of Section \ref{ss:coha.A1}), we have
	\begin{equation}
		\label{eqn:EE.is.PA1}
		\EE(z) = \sum_{r,k \geq 0} (-z)^r q^{r^2/2 + k} \dim\left( (\coha_{r})_k \right).
	\end{equation}
We will use the connection between \eqref{eqn:q.dilog.defn} and \eqref{eqn:EE.is.PA1} to establish our main theorem. In particular, we will need the the major results of \cite{ja2018}, which we now restate for completeness.

First, we define the \emph{quantum algebra} $\A_Q$ of the quiver $Q$ to be the $\QQ(q^{1/2})$-algebra with vector space basis given by symbols $y_\gamma$, one for each dimension vector $\gamma \in D$, and subject to the relations
	\begin{equation}
		\label{eqn:qalg.relation}
	y_{\gamma_1+\gamma_2} = -q^{-\langle\gamma_1,\gamma_2\rangle/2}\,y_{\gamma_1}y_{\gamma_2}
	\end{equation}
for every $\gamma_1,\gamma_2 \in D$. In particular, we have that the symbols $\{y_{e_i}: i\in Q_0\}$ generate $\A_Q$ as an algebra. Let $\widehat{\A}_Q$ denote the \emph{completed quantum algebra} in which formal power series in the $y_\gamma$ symbols are allowed, but are subject to the same relation \eqref{eqn:qalg.relation}.

\begin{prop}[\cite{ja2018}, Theorem~4.2]
	\label{prop:EQ.factors}
Let $\cQ$ be an admissible, ordered, Dynkin subquiver partition. Further, assume that $\{e_1,\ldots,e_n\}$ is in ``head before tail'' order and that $\Phi_+(\cQ) = \{\beta_1,\ldots,\beta_r\}$ is in Reineke order. We have the identity
	\begin{equation}
		\label{eqn:EQ.factors}
			\EE(y_{e_1})\cdots\EE(y_{e_n}) = 
			\EE(y_{\beta_1}) \cdots \EE(y_{\beta_r}).
			\nonumber
	\end{equation}
which holds in the completed quantum algebra $\widehat{\A}_Q$. \qed
\end{prop}

Furthermore, we will need the following important computation in $\widehat{\A}_Q$, which is established in \cite{ja2018} \emph{en route} to Proposition \ref{prop:EQ.factors}.

\begin{prop}[\cite{ja2018}, Proposition~5.1]
		\label{prop:compute.codim}
With the same hypotheses as Proposition \ref{prop:EQ.factors}, fix a dimension vector $\gamma \in D$ and let $\kpm $ be a $\cQ$-partition. Consider the product \[y_{\beta_1}^{m_1} y_{\beta_2}^{m_2} \cdots y_{\beta_r}^{m_r} \in \A_Q\;\subset \widehat{\A}_Q.\]
We have
\begin{equation}
\label{eqn:positives.to.simples}
y_{\beta_1}^{m_1} y_{\beta_2}^{m_2} \cdots y_{\beta_r}^{m_r} = (-1)^{s_\kpm} \cdot q^{w_\kpm} \cdot y_{e_1}^{\gamma(1)}\cdots y_{e_n}^{\gamma(n)}
\nonumber
\end{equation}
where
\begin{gather}
	s_\kpm = \sum_{u=1}^r m_u\left(\sum_{i\in Q_0} d^i_u -1\right), \nonumber \\
	w_\kpm = \codim_{\CC}\left(\eta_m;\Rep_\gamma(Q)\right) 
		+ \frac{1}{2}\sum_{i\in Q_0} \gamma(i)^2 
		- \frac{1}{2}\sum_{u=1}^r m_u^2, \nonumber
\end{gather}
and we recall that the integers $d^i_u$ are determined by $\beta_u = \sum_{i\in Q_0} d^i_u e_i$ for each $\beta_u \in \Phi_+(\cQ)$. We further have 
	\begin{equation}
		\label{eqn:codim.eta.m}
	\codim_{\CC}\left(\eta_\kpm;\Rep_\gamma(Q)\right) = 
		\sum_{j\in[\ell]}
			\codim_{\CC}\left(\Omega_{\kpm^j}(Q^j);\Rep_{\gamma^j}(Q^j)\right),
	\end{equation}
where Equation \eqref{eqn:codim.eta.m} is the content of \cite[Proposition~2.5]{ja2018}. \qed
\end{prop}

\section{Structure theorems for CoHA}
	\label{s:structure}
	
Let $\cQ = \{Q^1,\ldots,Q^\ell\}$ be an admissible, ordered, Dynkin subquiver partition of $Q$, such that none of the subquivers $Q^j$ is an orientation of $E_8$. Recall that $\Phi_+(\cQ) = \Union_{j\in[\ell]}\Phi_+(Q^j)$ has a Reineke order \[\{\beta^1_1,\ldots,\beta^1_{r_1};\ldots;\beta^\ell_1,\ldots,\beta^\ell_{r_\ell}\}\] where $r_j = |\Phi_+(Q^j)|$. Further recall that in an abuse of notation, we let $r = \sum_{r\in[\ell]}r_j$ and also give the positive roots a second name, by writing the same Reineke order as $\{\beta_1,\ldots,\beta_r\}$. Moreover, recall that for each $u\in[r]$ (and uniquely determined $j$ and $k$), we have that \[\beta_u = \beta^j_{k} = \sum_{i\in Q^j_0} d^i_{u} e_i\] for some non-negative integers $d^i_{u}$. 

Now, for each $u \leftrightarrow (j,k)$ fix a vertex $i(u) = i(j,k) = i$, for which $d^i_{u} = 1$. We fix this choice of $i$ in the sequel, and when there is no confusion we will simply write $i$ instead of $i(u)$ or $i(j,k)$. Every positive root of every Dynkin root system admits such an $i$ except for the longest root of $E_8$, which explains our exclusion.

Now, we let $\cA_{\beta_u}$ denote the subalgebra of $\coha$ generated by the set, cf.~\cite[Section~2.8]{mkys2011} and \cite[Definition~11.1]{rr2013},
	\begin{equation}
		\label{eqn:bar.cA}
	\overline{\cA}_{\beta_u} = \left\{ f(\omega_{i,1}) : f\in\QQ[x]\right\} \subset \coha_{\beta_u}
	\nonumber
	\end{equation}

\begin{prop}
	\label{prop:cA.structure}
For each vertex $i\in Q_0$, we have $\overline{\cA}_{e_i} = \coha_{e_i}$. More generally, for every $\beta\in\Phi_+(\cQ)$ we have $\cA_\beta \iso \coha(A_1)$. That is, each $\cA_\beta$ is isomorphic to the CoHA of a quiver with a single vertex and no arrows described in Section \ref{ss:coha.A1}. 
\end{prop}

\begin{proof}
The isomorphism $\cA_\beta \iso \coha(A_1)$ is achieved by sending $\omega_{i,1}^p \in \coha_{\beta}$ to the element $\psi_p$ from Section \ref{ss:coha.A1}. For simple roots $\beta = e_i$, the claims follow by noticing that the numerator factors (parameterized by arrows) in the multifactor multiplication will all be $1$, and the only denominator factors (parameterized by vertices) which are not $1$ are those involving the vertex $i$. 

For general positive roots, the multifactor multiplication formula appears to involve variables $\omega_{i',j}$ for $i'\neq i$, but organization of the terms reveals that dependence on these variables cancels thanks to supersymmetry, and the result is again identical to \eqref{eqn:schur}. 
\end{proof}

One consequence of Proposition \ref{prop:cA.structure} is that the elements of $\cA_{\beta}$ are the polynomials in $\coha_{m\beta}$, for all $m \in \{0,1,2,\ldots\}$, which only depend on the variables $\{\omega_{i,b} :1\leq b\leq m\}$ at vertex $i\in Q_0$. 

\begin{remark}
The freedom in the choice of $i(u)$ is evident already in the seminal work of Kontsevich--Soibelman in the example of $A_2$. Explicitly, the algebra $\curly{H}^{(i)}$ for $i=1$ (respectively $i=2$) in \cite[Section~2.8, Proposition~1]{mkys2011} is our algebra $\cA_{e_1+e_2} \subset \coha$ for $i(u)=1$ (resp.~$i(u)=2$) with $\beta_u = e_1+e_2$.
\end{remark}

With these definitions and conventions, we have the following structure theorem for $\coha$. The proof will be given in Section \ref{s:main.thm.pf}.
	
\begin{thm}
	\label{thm:main}
Suppose that $\cQ=\{Q^1,\ldots,Q^\ell\}$ is an admissible, ordered, Dynkin subquiver partition of the acyclic quiver $Q$ such that none of the subquivers $Q^j$ is an orientation of $E_8$. Let the associated positive roots be written in a Reineke order, i.e., $\Phi_+(\cQ) = \{\beta_1,\ldots,\beta_r\}$. Then the $*$-multiplication induces an isomorphism
\begin{equation}
	\label{eqn:coha.structure}
\cA_{\beta_1} \tensor \cdots \tensor \cA_{\beta_r} \stackrel{\iso}{\longrightarrow} \coha
\end{equation}	
with the products taken in order from left to right.
\end{thm}

We first comment that when $Q$ is a Dynkin quiver, our result interpolates between the two isomorphisms established by Rim\'anyi \cite[Theorem~11.3]{rr2013}, which we restate below.

\begin{thm}[Rim\'anyi, \cite{rr2013}]
Let $Q$ be a Dynkin quiver (but not an orientation of $E_8$), let $e_1,\ldots,e_n$ denote its simple roots in ``head before tail'' order, and let $\beta_1,\ldots,\beta_N$ denote its positive roots in Reineke order. Then the $*$-multiplication induces isomorphisms
\begin{equation}
	\label{eqn:st.bas.iso}
	\cA_{e_1} \tensor \cdots \tensor \cA_{e_n} \stackrel{\iso}{\longrightarrow} \coha
\end{equation}
\begin{equation}
	\label{eqn:pos.rts.iso}
	\cA_{\beta_1} \tensor \cdots \tensor \cA_{\beta_N} \stackrel{\iso}{\longrightarrow} \coha.
\end{equation}
The isomorphism \eqref{eqn:st.bas.iso} still holds when $Q$ is an orientation of $E_8$. \qed
\end{thm}

In particular, when $Q$ is Dynkin (but not an orientation of $E_8$) and $\cQ = \{Q\}$ we have that our isomorphism \eqref{eqn:coha.structure} coincides with \eqref{eqn:pos.rts.iso}. On the other extreme, when $\cQ$ consists of $n$ subquivers, each of which is a single vertex of the Dynkin quiver $Q$, we have that our isomorphism \eqref{eqn:coha.structure} coincides with \eqref{eqn:st.bas.iso}.

We observe that even for general acyclic $Q$ (not necessarily Dynkin), the fact that \eqref{eqn:st.bas.iso} is an isomorphism follows immediately from the definitions. 

\begin{lem}
	\label{lem:simples.iso}
When $\cQ$ consists of the subquivers $Q^i_0 = \{i\}$ for $i\in [n]$, ordered in ``head before tail'' ordering, then \eqref{eqn:coha.structure} is an isomorphism.
\end{lem}

\begin{proof}
For $f_i(\omega_{i,1},\ldots,\omega_{i,\gamma(i)}) \in \cA_{e_i}$ we have, say from the multifactor version of \eqref{eqn:coha.two.factor} or from \eqref{eqn:multi.mult.as.pi}, that
	\begin{equation}
		\label{eqn:st.basis.iso.proof}
		f_1 * \cdots * f_n = \prod_{i\in Q_0} f_i(\omega_{i,1},\ldots,\omega_{i,\gamma(i)}) \in \coha_\gamma
	\end{equation} 
from whence the result follows.
\end{proof}

Furthermore, our main theorem can be restated as follows. 

\begin{cor}
	\label{cor:dynkin.decomp}
With the same hypotheses as Theorem \ref{thm:main}, the mapping induced by the multifactor $*$-multiplication
	\begin{equation}
	\coha(Q^1) \tensor \cdots \tensor \coha(Q^\ell) \stackrel{\iso}{\longrightarrow} \coha(Q)
	\nonumber
	\end{equation}
is an isomorphism. That is, $\coha = \coha(Q)$ can be decomposed into tensor factors, each isomorphic to the CoHA of a Dynkin quiver. 
\end{cor}

\begin{proof}
Assuming the truth of \eqref{eqn:coha.structure}, the associativity of the $*$-multiplication and our convention on ordering allows us to group the factors of \eqref{eqn:coha.structure} according to subquivers. Since each $Q^j$ is Dynkin, the isomorphism of \eqref{eqn:pos.rts.iso}---which we recall is a special case of \eqref{eqn:coha.structure}---implies the result.
\end{proof}

It was established in \cite[Theorem~10.1]{rr2013} that the quiver polynomials are important structure constants in the CoHA of Dynkin quivers. On the one hand, our Proposition \ref{prop:eta.m.prod.of.1s} generalizes this property to the quiver stratum $\eta_\kpm$. On the other hand, the following result further relates the fundamental class of $\overline{\eta}_\kpm$ directly to Dynkin quiver polynomials. We note that the quiver polynomials have a rich history in their own right, and exhibit interesting geometric and combinatorial properties; see e.g., \cite{ab2008,rkakjr2019} and references therein.

\begin{cor}
	\label{cor:dynkin.orbit.mult}
For every $\cQ$-partition $\kpm$, performing multiplications of quiver polynomials according to the isomorphism of Corollary \ref{cor:dynkin.decomp} gives 
	\[
	\left( \left[ \overline{\Omega_{\kpm^1}(Q^1)} \right] \in \coha_{\gamma^1}(Q^1)\right)
		* \cdots * 
		\left( \left[ \overline{\Omega_{\kpm^\ell}(Q^\ell)} \right] \in \coha_{\gamma^\ell}(Q^\ell) \right)
			= \left[ \overline{\eta}_\kpm \right] \in \coha_\gamma(Q).
	\]
\end{cor}

\begin{proof}
By Proposition \ref{prop:eta.m.prod.of.1s} (or by \cite[Theorem~10.1]{rr2013}), we see that for each $j\in[\ell]$ we have
	\[ 
	\left( 1\in \coha_{m^j_1 \beta^j_1} \right) * \cdots *
		\left( 1 \in \coha_{m^j_{r_j} \beta^j_{r_j}} \right) = 
		\left[ \overline{\Omega_{\kpm^j}(Q^j)} \right].
	\]
By virtue of our ordering on the roots $\Phi_+(\cQ)$, the result follows by combining the above with Corollary \ref{cor:dynkin.decomp} and Proposition \ref{prop:eta.m.prod.of.1s}.
\end{proof}

\begin{remark}
Given that the result of Proposition \ref{prop:multifactor.pushforward.mult} does not depend on the subquivers $Q^j$ being Dynkin, it seems plausible that one can drop the Dynkin assumption from the subquivers $Q^j$ in Corollary \ref{cor:dynkin.decomp}. Unfortunately, in so doing, one loses the geometric connection to quiver polynomials given by Corollary \ref{cor:dynkin.orbit.mult}.
\end{remark}

\section{Proof of the main theorem}
	\label{s:main.thm.pf}

We have already shown that \eqref{eqn:coha.structure} is an isomorphism in the case that $\cQ$ consists only of subquivers with a single vertex, i.e., Lemma \ref{lem:simples.iso}. This section is dedicated to the proof that \eqref{eqn:coha.structure} is an isomorphism for all other choices of admissible, ordered, Dynkin subquiver partitions $\cQ$ (for which each subquiver $Q^j$ is not an orientation of $E_8$).

\subsection{Equivariant geometry of Dynkin quivers}
	\label{ss:eq.geom.dynkin}

We first consider an aside into the equivariant geometry of Dynkin quiver orbits. In particular, we recall an amalgam of several important results and constructions from \cite{lfrr2002.duke}, \cite{rr2013}, and \cite{ja2018}. To avoid confusion with our fixed acyclic quiver $Q$, for the moment we will let $R$ denote a fixed Dynkin quiver. 

Given any dimension vector $\gamma$ for $R$, and a Kostant partition $\kpm \kp \gamma$, recall we have the associated orbit $\Omega_\kpm(R) \subset \Rep_\gamma(R)$. Assume $R_0 = [n]$ and write $\Phi_+(R) = \{\beta_1,\ldots,\beta_r\}$ with $\beta_u = \sum_{i\in R_0} d^i_u \, e_i$. Observe that $\#\Phi_+(R)=r$ and so we can write $\kpm = (m_1,\ldots,m_r)$. 

Following Rim\'anyi (see the proof of \cite[Lemma~11.4]{rr2013}), we define a system of sets of non-negative integers $Y_{i,u,v}$ with $i\in[n]$, $u\in[r]$, and $v\in[m_u]$ (where if $m_u = 0$, then $Y_{i,u,v} = \emptyset$) as follows. We require that $\# Y_{i,u,v} = d^i_u$ and furthermore that for each $i\in Q_0$ the \emph{disjoint} union
	\[
	Y_{i,1,1}\union \cdots \union Y_{i,1,m_1} \union Y_{i,2,1} \union \cdots \union Y_{i,2,m_2} 
		\union \; \cdots \; \union Y_{i,r,1} \union \cdots \union Y_{i,r,m_r}
	\]
is equal to $\{1,\ldots,\gamma(i)\}$ \emph{in this order} when the elements are read off from left to right; i.e., we have \[Y_{i,u,v} = \{(v-1)d^i_u + 1, \ldots, v\,d^i_u\}.\] 

This system of sets determines a distinguished point $X_\kpm \in \Omega_\kpm(R) \subset \Rep_\gamma(R)$ as follows. Using the categorical equivalence between modules over the path algebra $\CC R$ and quiver representations, we recall that the indecomposable $\CC R$-modules are in one-to-one correspondence with positive roots and that the Krull--Schmidt theorem for path algebras implies that each quiver representation can be written uniquely as a direct sum of indecomposables, see e.g., \cite[Chapter~1]{hdjw2017}. Let $M_{\beta}$ denote the indecomposable representation corresponding to the positive root $\beta$.
	
Let $b_{i,1},\ldots,b_{i,\gamma(i)}$ be the standard basis of $\CC^{\gamma(i)}$. For each $u\in[r]$ and $v\in [m_u]$, let $M_{u,v}$ be an indecomposable $\CC R$-module isomorphic with $M_{\beta_u}$ and determined using the vectors $b_{i,j}$ for $j\in Y_{i,u,v}$. Set $X_\kpm$ to be the quiver representation corresponding to the $\CC R$ module $\Dirsum_{u\in[r]} \Dirsum_{v\in[m_u]} M_{u,v}$.

\begin{example}
	\label{ex:Dynkin.Y.sets}
We let $R$ be the equioriented $A_3$ quiver $1 \leftarrow 2 \leftarrow 3$ of previous examples. We take the same Reineke order on the positive roots from Examples \ref{ex:Reineke.order.A3} and \ref{ex:Orbit.A3}, and the same Kostant partition as \ref{ex:Orbit.A3}. That is, we take
	\begin{align*}
	\beta_1 & = e_3 & \beta_2 & = e_2+e_3 & \beta_3 & = e_2 \\
	\beta_4 & = e_1 + e_2 + e_3 & \beta_5 &= e_1 + e_2 & \beta_6 &= e_1
	\end{align*}
and
	\begin{align*}
	m_1 & = 0 &m_2 & = 2 & m_3 & = 1 & m_4 & = 1 & m_5 & =0 & m_6 &=1.
	\end{align*}
We depict the sets $Y_{i,u,v}$ below, as well as the point $X_\kpm$. Observe that we can see a version of the lacing diagram from Example \ref{ex:Orbit.A3}.

\begin{center}
	\begin{minipage}[t]{0.4\textwidth}
	\begin{center}
	\underline{The $Y_{i,u,v}$ sets}
	\end{center}
	\vspace{0.2\baselineskip}
	\begin{align*}
	Y_{1,2,1} & =\{~\} & Y_{2,2,1} &=\{1\} & Y_{3,2,1} = \{1\} \\
	Y_{1,2,2} & =\{~\} & Y_{2,2,2} &=\{2\} & Y_{3,2,2} = \{2\} \\
	Y_{1,3,1} & =\{~\} & Y_{2,3,1} &=\{3\} & Y_{3,3,1} = \{~\} \\	
	Y_{1,4,1} & =\{1\} & Y_{2,4,1} &=\{4\} & Y_{3,4,1} = \{3\} \\
	Y_{1,6,1} & =\{2\} & Y_{2,6,1} &=\{~\} & Y_{3,6,1} = \{~\} \\
	\end{align*}
	\end{minipage}
	\begin{minipage}{0.1\textwidth}
	\[
	\implies
	\]
	\end{minipage}
	\begin{minipage}[t]{0.4\textwidth}
	\begin{center}
	\underline{The quiver representation $X_\kpm$}
	\end{center}
	\vspace{0.8\baselineskip}
	\[\begin{diagram}[width=2.5em,height=1.5em]
	~		& 		& b_{2,1}	& \lTo	& b_{3,1}	\\
	~		&		& b_{2,2}	& \lTo	& b_{3,2}	\\
	~		&		& b_{2,3}	& 		& ~			\\
	b_{1,1}	& \lTo	& b_{2,4}	& \lTo	& b_{3,3}	\\
	b_{1,2}	&		& ~			&		& ~
	\end{diagram}\]
	\end{minipage}
\end{center}

In the depiction of $X_\kpm$ on the right, observe that when reading from top to bottom we see
	\begin{itemize}
	\item $m_1 = 0$ isomorphic copies of the indecomposable representation $M_{\beta_1} = M_{e_3}$;
	\item $m_2 = 2$ isomorphic copies of the indecomposable representation $M_{\beta_2} = M_{e_2+e_3}$, determined by basis vectors $b_{2,v}$ and $b_{3,v}$ for $v=1,2$;
	\item $m_3 = 1$ isomorphic copies of the indecomposable representation $M_{\beta_3} = M_{e_2}$, determined by $b_{2,3}$;
	\item $m_4 = 1$ isomorphic copy of the indecomposable representation $M_{\beta_4} = M_{e_1+e_2+e_3}$, determined by $b_{1,1}$, $b_{2,4}$, and $b_{3,3}$; 
	\item $m_5 = 0$ isomorphic copies of the indecomposable representation $M_{\beta_5} = M_{e_1+e_2}$;
	\item $m_6 = 1$ isomorphic copy of the indecomposable representation $M_{\beta_6} = M_{e_1}$; determined by $b_{1,2}$.
	\end{itemize}
$X_\kpm$ has the following matrix representation (assuming the standard bases $\{b_{i,j}\}$ at each vertex)
	\begin{equation}
	X_\kpm = \left( 
		\left[\begin{array}{cccc}
		0 & 0 & 0 & 1 \\
		0 & 0 & 0 & 0
		\end{array}\right] ,
		\left[\begin{array}{ccc}
		1 & 0 & 0 \\
		0 & 1 & 0 \\
		0 & 0 & 0 \\
		0 & 0 & 1
		\end{array}\right]
		\right)
					\in \Rep_{(2,4,3)}(R). \nonumber
	\end{equation}
\end{example}

Let $G_\kpm$ denote the isotropy subgroup of $\Omega_\kpm \subset \Rep_\gamma(R)$. Up to isomorphism, $G_\kpm$ is the stabilizer in $\GGLL_\gamma$ of the point $X_\kpm$ and moreover, Feher--Rim\'anyi proved that up to homotopy we have $G_\kpm \hmtpc \prod_{u\in[r]} \GL(\CC^{m_u})$ \cite[Proposition~3.6]{lfrr2002.duke}. Further, Feher--Rim\'anyi study the restriction mapping $\iota^*_\kpm: \coho^\bullet_{\GGLL_\gamma}(\Rep_\gamma(R)) \to \coho^\bullet_{\GGLL_\gamma}(\Omega_\kpm)$ induced by the inclusion $\iota_\kpm:\Omega_\kpm \includes \Rep_\gamma(R)$. This mapping can be identified with a map $\coho^\bullet(B\GGLL_\gamma) \to \coho^\bullet(BG_\kpm)$. We already know that the source of this latter map is a polynomial ring via \eqref{eqn:H.gamma.polynomials}, but the homotopy type of $G_\kpm$ means that we can also write the target of this map as a polynomial ring. The mapping is then given by
\begin{equation}
	\label{eqn:restriction.dynkin}
\begin{aligned}
	\iota^*_\kpm : 
		\Tensor_{i\in Q_0} \QQ[\omega_{i,1},\ldots,\omega_{i,\gamma(i)}]^{\Sym_{\gamma(i)}}
		& \longrightarrow 
		\Tensor_{u\in [r]} \QQ[\tau_{u,1},\ldots,\tau_{u,m_u}]^{\Sym_{m_u}}  \\
		\omega_{i,k}\qquad\qquad  & \longmapsto \qquad\qquad\tau_{u,v} \quad \text{if $k\in Y_{i,u,v}$}.
\end{aligned}
\end{equation}
For more details on the mapping above, see \cite[Section~3]{lfrr2002.duke} and \cite[Section~11]{rr2013}. In particular, we note that our choice of the sets $Y_{i,u,v}$ amounts to an ordering on the variables above. However, since we consider only supersymmetric polynomials on the lefthand and righthand sides, any other allowed ordering produces the same mapping $\iota^*_\kpm$. We will see in the proof of Proposition \ref{prop:graded.mult.injective} why the specific choice of the sets $Y_{i,u,v}$ is good for our purposes.

\subsection{Injectivity}
	\label{ss:inj}

We now generalize the constructions of the previous subsection to our setting of the acyclic quiver $Q$, the (ordered and admissible) Dynkin subquiver partition $\cQ$, and a $\cQ$-partition $\kpm$. In particular, for each Dynkin subquiver $Q^j \in \cQ$, we have an associated Kostant partition $\kpm^j = (m^j_1,\ldots,m^j_{r_j}) \kp \gamma^j$, and this determines a system of sets $Y^j_{i,u',v'}$ with $i\in Q^j_0$, $u'\in[r_j]$, and $v'\in[m^j_{u'}]$. In turn, we have distinguished points $X_{m^j} \in \Omega_{\kpm^j}(Q^j) \subset \Rep_{\gamma^j}(Q_j)$. As in \cite{ja2018} we define a \emph{normal locus} associated to $\kpm$ to be the following subspace of $\eta_\kpm$
\begin{equation}
	\label{eqn:normal.locus.defn}
	\nu_\kpm = 
		\{ (\phi_a)_{a\in Q_1} \in \Rep_{\gamma} : 
		\forall j\in[\ell],~(\phi_a)_{a\in Q^j_1} = X_{\kpm^j} \}.
		\nonumber
\end{equation}
Since for each $j\in[\ell]$ we have $X_{\kpm^j} \in \Omega_{\kpm^j}(Q^j)$ we indeed have $\nu_\kpm \subset \eta_\kpm$. The normal locus will play the role of the distinguished points $X_\kpm$. Although $\nu_\kpm$ is not a singleton point, since each $X_{\kpm^j}$ \emph{is} a singleton point in $\Rep_{\gamma^j}(Q^j)$, we have a natural identification
\begin{equation}
	\label{eqn:normal.locus.homeomorphic}
	\nu_\kpm \homeo \Dirsum_{a\in Q_1 \setminus \cQ_1} \Hom(\CC^{\gamma(ta)},\CC^{\gamma(ha)})
\end{equation}
so $\nu_\kpm$ is homeomorphic to an (equivariantly) contractible vector space. We can also generalize the notion of the isotropy subgroup to this context and set $G_\kpm = \{ g\in \GGLL_\gamma : g\cdot \nu_\eta = \nu_\eta\}$. From this definition, it follows that 
	\begin{equation}
		\label{eqn:normal.locus.isotropy}
		G_\kpm 	\iso \prod_{j \in [\ell]} G_{\kpm^j} 
				\hmtpc \prod_{j\in [\ell]} \prod_{u' \in [r_j]} \GL(\CC^{m^j_{u'}}),
		\nonumber
	\end{equation}
which generalizes the Feher--Rim\'anyi result to the present context; see \cite[Proposition~6.1]{ja2018}. Hence, combining the analyses of \cite[Section~3]{lfrr2002.duke} and \cite[Section~6]{ja2018} we further obtain a generalized restriction mapping $\iota^*_\kpm : \coho^*(B\GGLL_\gamma) \to \coho^*(BG_\kpm)$ induced by the inclusion $\iota_\kpm:\eta_\kpm \includes \Rep_\gamma$ and given by
	\begin{equation}
		\label{eqn:general.restriction}
	\begin{aligned}
	\iota^*_\kpm : 
		\Tensor_{i\in Q_0} \QQ[\omega_{i,1},\ldots,\omega_{i,\gamma(i)}]^{\Sym_{\gamma(i)}}
		& \longrightarrow 
		\Tensor_{j\in[\ell]}
		\Tensor_{u'\in [r_j]} \QQ[\tau_{j,u',1},\ldots,\tau_{j,u',m^j_{u'}}]^{\Sym_{m^j_{u'}}}  \\
		\omega_{i,k}\qquad\qquad  & \longmapsto \qquad\qquad\tau_{j,u',v'} \quad \text{if $k\in Y^j_{i,u',v'}$};
	\end{aligned}
	\nonumber
	\end{equation}

Let $\varepsilon_\kpm \in \coho^\bullet(BG_\kpm)$ denote the Euler class of the normal bundle to $\eta_\kpm$ restricted to $\nu_\kpm$. Further, for each $j$ we have isotropy groups $G_{\kpm^j} \leq \GGLL_{\gamma^j}$ and we can denote the Euler classes of the normal bundles to $\Omega_{\kpm^j} \subset \Rep_{\gamma^j}(Q^j)$ at $X_{\kpm^j}$ by $\varepsilon_{\kpm^j} \in \coho^\bullet(BG_{\kpm^j})$. The classes $\varepsilon_{\kpm^j}$ have natural inclusions into $\coho^\bullet(BG_\kpm)$, via the K\"unneth isomorphism $\coho^\bullet(BG_\kpm) \iso \Tensor_{j\in[\ell]} \coho^\bullet(BG_{\kpm^j})$. We will need the following observation in the sequel.

\begin{prop}
	\label{prop:euler.normal.appears}
$\iota^*_\kpm( [\overline{\eta}_\kpm] ) = \prod_{j\in[\ell]} \varepsilon_{\kpm^j} = \varepsilon_\kpm$.
\end{prop}

\begin{proof}
The result generalizes an observation of \cite[Section~11]{rr2013} in the context of Dynkin quivers, where $\iota_{\kpm^j}^*([\overline{\Omega}_{\kpm^j}]) = \varepsilon_{\kpm^j}$. The same logic applies here, because the class of a variety (in this case $\overline{\eta}_\kpm$) restricted to the smooth points ($\nu_\kpm$ is a subset of the smooth points) is exactly the Euler class of the normal bundle. We further observe that $\varepsilon_\kpm$ can be realized as the product of the clases $\varepsilon_{\kpm^j}$ because there are no extra contributions from the arrows $a\in Q_1\setminus \cQ_1$ since the identification \eqref{eqn:normal.locus.homeomorphic} implies that $\nu_\kpm$ has full dimension there and, in fact, is the whole space $\Hom(\CC^{\gamma(ta)},\CC^{\gamma(ha)})$.
\end{proof}

For $\beta\in\Phi_+(\cQ)$ and $m \geq 0$, let $\cA_{\beta,m} = \cA_\beta \intersect \coha_{m\beta}$. Consider the multi-factor multiplication 
	\[
		\mu_\kpm:\cA_{\beta_1,m_1} \tensor \cdots \tensor \cA_{\beta_r,m_r} 
			\longrightarrow \coha_\gamma.
	\]

\begin{prop}
	\label{prop:mu.kpm.injective}
$\mu_\kpm$ is injective for all choices of $\cQ$-partition $\kpm = (m_1,\ldots,m_r)$.
\end{prop}

To prove Proposition \ref{prop:mu.kpm.injective}, we consider the composition of $\mu_\kpm$ with the restriction mapping $\iota^*_{\kpm}$. Let $f_u(\omega_{i(u),1},\ldots,\omega_{i(u),m_u}) \in \cA_{\beta_u,m_u}$ for each $u\in[r]$. Moreover, we rename the variables $\tau_{j,u',v'}$ as follows. Recall the two ways in which we name the positive roots (see Section \ref{s:structure}), and hence entries of a $\cQ$-partition. Namely $\kpm$ is equal to the sequence of non-negative integers
\[
(m_1,\ldots,m_r) = (m^1_1,\ldots,m^1_{r_1},\,\ldots\,,m^\ell_1,\ldots,m^\ell_{r_\ell})
\]
where $r_j=\#\Phi_+(Q^j)$, $r = r_1 + \cdots + r_\ell = \#\Phi_+(\cQ)$, and both sequences above are \emph{in the same order}. For fixed $j$ and $u'$, we have the $m^j_{u'}$ variables $\{\tau_{j,u',1},\ldots,\tau_{j,u',m^j_{u'}}\}$. If the entry $m^j_{u'}$ (in the second sequence) corresponds to $m_u$ (in the first sequence), then we replace the variables $\{\tau_{j,u',1},\ldots,\tau_{j,u',m^j_{u'}}\}$ by variables $\{t_{u,1},\ldots,t_{u,m_u}\}$. In particular, since $m^j_{u'} = m_u$, these sets of variables indeed have the same number of elements. Further, this same correspondence extends to an identification of the $Y$-sets. In particular, we rename $Y^j_{i,u',v'}$ as $Y_{i,u,v'}$.

\begin{prop}
	\label{prop:graded.mult.injective}
The map $\iota_\kpm^{*} \compose \mu_\kpm : \cA_{\beta_1,m_1} \tensor \cdots \tensor \cA_{\beta_r,m_r} \to \coho^\bullet(BG_\kpm)$ is given by
\begin{multline}
	\label{eqn:graded.mult.image}
f_1(\omega_{i(1),1},\ldots,\omega_{i(1),m_1}) \tensor 
	\cdots \tensor f_r(\omega_{i(r),1},\ldots,\omega_{i(r),m_r}) \longmapsto \\
f_1(t_{1,1},\ldots,t_{1,m_1}) \cdots f_r(t_{r,1},\ldots,t_{r,m_r}) \cdot \varepsilon_\kpm.
\end{multline}
\end{prop}

\begin{proof}
The product $f_1 * \cdots * f_r$ is given by the multi-factor version of \eqref{eqn:coha.two.factor}, which will have many terms. On the other hand, the formula \eqref{eqn:multi.mult.as.pi} says that this sum is really the torus equivariant localization formula for the $\pi_*$ mapping and, in particular, each term corresponds to a torus fixed point in $\FFll_{m_1\beta_1,\ldots,m_r\beta_r} \times \Rep_\gamma$. The content of Proposition \ref{prop:consist.subset.desing} is that $\pi$ gives a resolution of the singularities of $\overline{\eta}_\kpm$, in particular because it is a product of mappings which resolve the singularities of each Dynkin orbit closure $\overline{\Omega_{\kpm^j}(Q^j)}$. 

Now, we claim there is \emph{only one} torus fixed point in $\Sigma = \Sigma_{m_1\beta_1,\ldots,m_r\beta_r}(\cQ)$ over $\nu_\kpm$. Indeed, this is true over each $X_{\kpm^j}$ in the factors of $\Sigma$ corresponding to $Q^j$. Combining this observation with the fact that \eqref{eqn:multi.mult.as.pi} implies $\pi_*$ does not ``see'' any of the arrows in $Q_1\setminus \cQ_1$ proves the claim. It follows that when applying the restriction $\iota^*_\kpm$ to the localization formula, \emph{only one} of the terms survives! All others map to zero in $\coho^\bullet(BG_\kpm)$.

Further recall that each term in the localization sum corresponds to choices of subsets of $\{1,\ldots,\gamma(i)\}$ and one checks that the surviving term must correspond to the choice of subsets $\Union_{v\in[m_u]} Y_{i,u,v}$. From this, we can conclude that the $\iota^*_\kpm$ sends the surviving term to
	\[
	f_1(t_{1,1},\ldots,t_{1,m_1})\,\cdots\,f_r(t_{r,1},\ldots,t_{r,m_r})\cdot g(t_{u,v})
	\]
where $g$ is a rational function whose numerator and denominator are both products of linear factors of the form $t_{u,v} - t_{u',v'}$. Moreover, since $g$ comes from the part of the localization formula determined by the pushforward along the flag manifold only, it is independent of the $f_u$ polynomials! We can thus determine its value by choosing $f_u = 1$ for all $u\in[r]$. Doing so, we see that
	\[
	g = \iota^*_\kpm(1*\cdots*1) = \iota^*_\kpm([\overline{\eta}_\kpm]) = \varepsilon_\kpm
	\]
where the second equality is Proposition \ref{prop:eta.m.prod.of.1s} and the third equality is Proposition \ref{prop:euler.normal.appears}.
\end{proof}

\begin{example}
	\label{ex:injectivity.proof}
Here we illustrate the key points of the results above. Again we let $Q$ be the equioriented $A_3$ quiver $1 \leftarrow 2 \leftarrow 3$. However, this time we consider the admissible, ordered, Dynkin subquiver partition $\cQ = \{Q^1,Q^2\}$ where $Q^1$ is the type $A_1$ quiver at vertex $1$, and $Q^2$ is the type $A_2$ subquiver $2 \leftarrow 3$. A Reineke order (actually unique in this example) on the corresponding positive roots is
\begin{align*}
\beta_1 & = e_1 & \beta_2 &=e_3 & \beta_3 & = e_2 + e_3 & \beta_4 & = e_2
\end{align*}
and we will take the $\cQ$-partition $m_1 = 2$, $m_2 = m_3 = m_4 = 1$ of the dimension vector $\gamma = (2,2,2)$. The $Y_{i,u,v}$ sets are
\begin{align*}
	Y_{1,1,1} & = \{1\} & Y_{2,1,1} & = \{~\} & Y_{3,1,1} & = \{~\} \\
	Y_{1,1,2} & = \{2\} & Y_{2,1,2} & = \{~\} & Y_{3,1,2} & = \{~\} \\
	Y_{1,2,1} & = \{~\} & Y_{2,2,1} & = \{~\} & Y_{3,2,1} & = \{1\} \\
	Y_{1,3,1} & = \{~\} & Y_{2,3,1} & = \{1\} & Y_{3,3,1} & = \{2\} \\
	Y_{1,4,1} & = \{~\} & Y_{2,4,1} & = \{2\} & Y_{3,4,1} & = \{~\}. \\
\end{align*}
Thus, the normal locus $\nu_{2,1,1,1}$ looks like
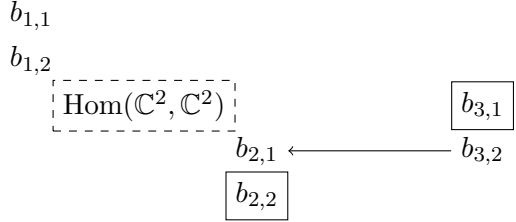
\begin{center}
	\begin{tikzpicture}
	\node (b11) at (0,2.4) {$b_{1,1}$};
	\node (b12) at (0,1.8) {$b_{1,2}$};
	\node (b21) at (3,0.6) {$b_{2,1}$};
	\node[draw, shape=rectangle] (b22) at (3,0.0) {$b_{2,2}$};
	\node[draw, shape=rectangle] (b31) at (6,1.2) {$b_{3,1}$};
	\node (b32) at (6,0.6) {$b_{3,2}$};
	
	\draw[->] (b32)--(b21);
	
	\node[draw, shape=rectangle,dashed] at (1.5,1.2) {$\Hom(\CC^2,\CC^2)$};
	\end{tikzpicture}
\end{center}
that is, \[\nu_{2,1,1,1} = \left\{ \left( \left[\begin{array}{cc} z_{11} & z_{12} \\ z_{21} & z_{22} \end{array}\right] , \left[\begin{array}{cc} 0 & 1 \\ 0 & 0 \end{array}\right] \right) \in \Rep_{2,2,2}(Q): z_{i,j} \in \CC \right\}.\] The vectors boxed with solid lines in the diagram above are those which contribute to the normal directions to $\nu_{2,1,1,1}$ in $\Rep_{2,2,2}$; they span the cokernel and kernel of the map from vertex $3$ to vertex $2$. Now we consider the restriction mapping 
\begin{multline*}
	\iota^{*}_{2,1,1,1}:
	\QQ[\omega_{1,1},\omega_{1,2}]^{\Sym_2} \tensor \QQ[\omega_{2,1},\omega_{2,2}]^{\Sym_2} 
		\tensor \QQ[\omega_{3,1},\omega_{3,2}]^{\Sym_2} \\ 
	\longrightarrow \QQ[t_{1,1},t_{1,2}]^{\Sym_2} \tensor \QQ[t_{2,1}] 
							\tensor \QQ[t_{3,1}] \tensor \QQ[t_{4,1}]
\end{multline*}
which, from the $Y_{i,u,v}$ sets, we see is given by
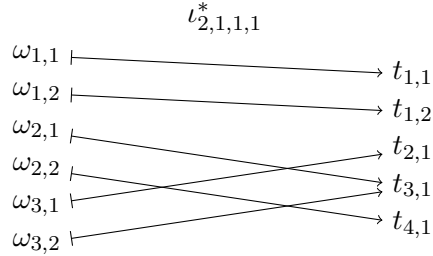
\begin{center}
	\begin{tikzpicture}
		\node (w11) at (0,2.5) {$\omega_{1,1}$};
		\node (w12) at (0,2) {$\omega_{1,2}$};
		\node (w21) at (0,1.5) {$\omega_{2,1}$};
		\node (w22) at (0,1) {$\omega_{2,2}$};
		\node (w31) at (0,.5) {$\omega_{3,1}$};
		\node (w32) at (0,0) {$\omega_{3,2}$};
		
		\node (t11) at (5,2.25) {$t_{1,1}$};
		\node (t12) at (5,1.75) {$t_{1,2}$};
		\node (t21) at (5,1.25) {$t_{2,1}$};
		\node (t31) at (5,0.75) {$t_{3,1}$};
		\node (t41) at (5,0.25) {$t_{4,1}$};
		
		\node (iota) at (2.5,3) {$\iota^*_{2,1,1,1}$};
		
		\draw [|->] (w11) -- (t11);
		\draw [|->] (w12) -- (t12);
		\draw [|->] (w21) -- (t31);
		\draw [|->] (w22) -- (t41);
		\draw [|->] (w31) -- (t21);
		\draw [|->] (w32) -- (t31);
	\end{tikzpicture}
\end{center}

Next, we choose $i(1) = 1$, $i(2) = 3$, $i(3) = 2$, and $i(4) = 2$ and let $f_u \in \cA_{\beta_u,m_u}$. In particular, $f_1 = f_1(\omega_{1,1},\omega_{1,2}) \in \coha_{2e_1}$, $f_2 = f_2(\omega_{3,1})\in\coha_{e_3}$, $f_3 = f_3(\omega_{2,1})\in \coha_{e_2+e_3}$, and $f_4 = f_4(\omega_{2,1}) \in \coha_{e_2}$. The multifactor multiplication gives
	\begin{multline*}
	(f_1*f_2*f_3*f_4)(\omega_{1,1},\omega_{1,2},\omega_{2,1},\omega_{2,2},\omega_{3,1},\omega_{3,2}) = \\
		f_1(\omega_{1,1},\omega_{1,2})f_2(\omega_{3,1})f_3(\omega_{2,1})f_4(\omega_{2,2}) 
		\frac{(\omega_{2,2}-\omega_{3,1})(\omega_{2,2}-\omega_{3,2})(\omega_{2,1}-\omega_{3,1})}
			{(\omega_{3,2}-\omega_{3,1})(\omega_{2,2}-\omega_{2,1})} + \clubsuit
	\end{multline*}
where $\clubsuit$ consists of three other similar terms. The term written in full above is the one corresponding to the $Y_{i,u,v}$ sets; each of the other three terms has a factor of $(\omega_{3,2} - \omega_{2,1})$ occurring in the numerator. Hence applying $\iota^*_{2,1,1,1}$ above yields the single term
	\begin{multline}
		\label{eqn:after.restriction}
	f_1(t_{1,1},t_{1,2})f_2(t_{2,1})f_3(t_{3,1})f_4(t_{4,1}) 
		\frac{(t_{4,1}-t_{2,1})\cancel{(t_{4,1}-t_{3,1})(t_{3,1}-t_{2,1})}}
			{\cancel{(t_{3,1}-t_{2,1})(t_{4,1}-t_{3,1})}}  \\
	= f_1(t_{1,1},t_{1,2})f_2(t_{2,1})f_3(t_{3,1})f_4(t_{4,1}) \cdot (\underbrace{t_{4,1}-t_{2,1}}_{\varepsilon_{\kpm^2}}).
	\end{multline}
We see that $\varepsilon_{\kpm^2} = t_{4,1} - t_{2,1}$ by comparing to the boxed normal directions $b_{2,2}$ and $b_{3,1}$ in the previous diagram. Moreover, since $\varepsilon_{\kpm^1} = 1$, we have $\varepsilon_{2,1,1,1} = t_{4,1} - t_{2,1}$, which verifies the previously claimed results in this instance. 

Had we instead chosen $i(3)=3$, then we would have $f_3 \in \coha_{e_2+e_3}$ a polynomial in the variable $\omega_{3,1}$ and the term from $f_1 * f_2 * f_3 *f_4$ which will survive is
	\[
	f_1\left(\omega _{1,1},\omega _{1,2}\right) f_2\left(\omega _{3,1}\right) f_3\left(\omega _{3,2}\right) f_4\left(\omega _{2,2}\right)
	\frac{ \left(\omega _{2,2}-\omega _{3,1}\right) \left(\omega _{2,2}-\omega _{3,2}\right) \left(\omega _{2,1}-\omega _{3,1}\right) }{\left(\omega _{3,2}-\omega _{3,1}\right) \left(\omega _{2,2}-\omega _{2,1}\right) }.
	\]
Either way, observe that the result of applying $\iota_\kpm^*$ to the above agrees with \eqref{eqn:after.restriction}.
\end{example}

\begin{proof}[Proof of Proposition \ref{prop:mu.kpm.injective}]
We have that the target of the mapping $\iota^*_\kpm \compose \mu_\kpm$ is an integral domain; Proposition \ref{prop:graded.mult.injective} gives a formula which is a product of nonzero elements. Therefore $\iota^*_\kpm\compose\mu_\kpm$ is injective, and it follows that $\mu_\kpm$ is injective.
\end{proof}

\begin{lem}
	\label{lem:inj}
The mapping \eqref{eqn:coha.structure} is injective.	
\end{lem}

\begin{proof}
Since Proposition \ref{prop:mu.kpm.injective} holds for all choices of the $\cQ$-partition $\kpm$, it immediately implies the injectivity of \eqref{eqn:coha.structure}.
\end{proof}

\begin{remark}
The proof above generalizes the methods for an analogous result from \cite{rr2013}, and illustrates they are applicable in the acyclic setting.
\end{remark}

\begin{remark}
Observe it suffices for our injectivity argument that $\iota^*_\kpm\compose\mu_\kpm:f_1 \tensor \cdots \tensor f_r \mapsto f_1\cdots f_r\cdot g$ where $g$ is a nonzero polynomial. The fact that $g = \varepsilon_\kpm = \prod_{j} \varepsilon_{\kpm^j}$ is a ``bonus'' geometric fact.
\end{remark}

\subsection{Comparisons between Poincar\'e series}
	\label{ss:surj}
	
We adopt the notation $(\coha_\gamma)_k$ to denote the cohomological degree $2k$ part of the cohomology algebra $\coha_\gamma = \coho^\bullet_{\GGLL_\gamma}(\Rep_\gamma)$. As we have already noted, when $\cQ$ consists of the $n$ singleton vertices of $Q$ in ``head-before-tail'' order, then Lemma \ref{lem:simples.iso} says that the mapping
	\begin{equation}
	\cA_{e_1} \tensor \cdots \tensor \cA_{e_n} \to \coha
	\nonumber
	\end{equation}
induced by the $*$-multiplication is an isomorphism by \eqref{eqn:st.basis.iso.proof}. Thus we can see the shifted, twisted, Poincar\'e series for $\coha$ as an element of $\widehat{\A}_Q$ by considering
\begin{equation}
	\label{eqn:simp.rts.P.series}
	\EE(y_{e_1})\cdots\EE(y_{e_n}) = 
		\sum_{\gamma\in D} y_{e_1}^{\gamma(1)} \cdots y_{e_n}^{\gamma(n)} (-1)^{\sum_{i\in Q_0} \gamma(i)} q^{\sum_{i\in Q_0} \gamma(i)^2/2} \, \sum_{k\geq 0} q^k \dim\left((\coha_\gamma)_k\right).
\end{equation}
If we let $(\cA_{\beta_u,m_u})_{k_u}$ denote the degree $2k_u$ part of $\cA_{\beta_u,m_u}$ we also get \[(\cA_{\beta_u,m_u})_{k_u} \iso \coho^{2k_u}(B\GL(\CC^{m_u}))\] by Proposition \ref{prop:cA.structure} and Equation \eqref{eqn:coha.r.A1}. Thus given any admissible, ordered, Dynkin subquiver partition $\cQ$, we have that the shifted, twisted, Poincar\'e series of $\cA_{\beta_1} \tensor\cdots\tensor \cA_{\beta_r}$ is encoded by
\begin{multline}
	\label{eqn:pos.rts.P.series}
	\EE(y_{\beta_1})\cdots\EE(y_{\beta_r}) = \sum_{\kpm \in \ZZ^r_{\geq0} }
	 y_{\beta_1}^{m_1} \cdots y_{\beta_r}^{m_r} 
		(-1)^{\sum_{u\in[r]} m_u} q^{\sum_{u\in[r]} m_u^2/2} \times \left[\phantom{\int_0^1} \right. \\
		\left. \sum_{k\geq 0} q^k  \sum_{k_1 + \cdots + k_r = k} 
		 	\dim\left((\cA_{\beta_1,m_1})_{k_1}\right) 
			\cdots \dim\left((\cA_{\beta_r,m_r})_{k_r}\right) \right].
\end{multline}
To further manipulate the expression above, we have the following lemma.

\begin{lem}
	\label{lem:graded.mult.dim.shift}
If $\kpm$ is a $\cQ$-partition of $\gamma$, then for each $k$ we have
		\begin{multline*}
		\sum_{k= k_1 + \cdots + k_r} \dim\left((\cA_{\beta_1,m_1})_{k_1}\right) \cdots \dim\left( ( \cA_{\beta_r,m_r})_{k_r}\right) 
		= \dim(\cA_{\beta_1,m_1} * \cdots * \cA_{\beta_r,m_r})_{k + \codim(\eta_\kpm;\Rep_\gamma)}.
		\end{multline*}
\end{lem}

\begin{proof}
One can view this as a corollary to Proposition \ref{prop:graded.mult.injective}, since the degree of $\varepsilon_\kpm$ in Equation \eqref{eqn:graded.mult.image} corresponds exactly to $\codim(\eta_\kpm;\Rep_\gamma)$.
\end{proof}

Using the results of Proposition \ref{prop:compute.codim} and Lemma \ref{lem:graded.mult.dim.shift}, we see that \eqref{eqn:pos.rts.P.series} is further equal to
\begin{multline}
	\label{eqn:pos.rts.P.series.final}
	\sum_{\gamma \in D}  y_{e_1}^{\gamma(1)} \cdots y_{e_n}^{\gamma(n)} (-1)^{\sum_{i\in Q_0} \gamma(i)} q^{\sum_{i\in Q_0} \gamma(i)^2/2} 
	\left[
	\sum_{k\geq 0} q^k\, \sum_{\kpm\kp\gamma}\dim\left( (\cA_{\beta_1,m_1} * \cdots * \cA_{\beta_r,m_r})_k \right) \right].
\end{multline}

\begin{proof}[Completing the proof of Theorem \ref{thm:main}]
Proposition \ref{prop:EQ.factors} states that \eqref{eqn:simp.rts.P.series} and \eqref{eqn:pos.rts.P.series}, and hence \eqref{eqn:pos.rts.P.series.final}, are the same. Thus their comparison shows that both $\cA_{\beta_1}\tensor\cdots\tensor \cA_{\beta_r}$ and $\coha$ have the same Poincar\'e series. Since we already know that the mapping \eqref{eqn:coha.structure} is injective by Lemma \ref{lem:inj}, this is enough to guarantee that it must be an isomorphism.
\end{proof}


\bibliographystyle{amsalpha}
\bibliography{jmabib}

\def\cprime{$'$}
\providecommand{\bysame}{\leavevmode\hbox to3em{\hrulefill}\thinspace}
\providecommand{\MR}{\relax\ifhmode\unskip\space\fi MR }
\providecommand{\MRhref}[2]{%
  \href{http://www.ams.org/mathscinet-getitem?mr=#1}{#2}
}
\providecommand{\href}[2]{#2}
\begin{thebibliography}{BKTY05}

\bibitem[AB84]{marb1984}
M.~F. Atiyah and R.~Bott, \emph{The moment map and equivariant cohomology},
  Topology \textbf{23} (1984), no.~1, 1--28. \MR{721448}

\bibitem[ADF85]{saadf1980}
S.~Abeasis and A.~Del~Fra, \emph{Degenerations for the representations of a
  quiver of type {${A}_m$}}, J. Algebra \textbf{93} (1985), no.~2, 376--412.
  \MR{786760}

\bibitem[All14]{ja2014.ir}
J.~Allman, \emph{Grothendieck classes of quiver cycles as iterated residues},
  Michigan Math. J. \textbf{63} (2014), no.~4, 865--888. \MR{3286674}

\bibitem[All18]{ja2018}
\bysame, \emph{Interpolating factorizations for acyclic {D}onaldson--{T}homas
  invariants}, preprint, 2018.

\bibitem[AR18]{jarr2018}
J.~Allman and R.~Rim\'{a}nyi, \emph{Quantum dilogarithm identities for the
  square product of {A}-type {D}ynkin quivers}, Math. Res. Lett. \textbf{25}
  (2018), no.~4, 1037--1087. \MR{3882154}

\bibitem[BF99]{abwf1999}
A.~S. Buch and W.~Fulton, \emph{Chern class formulas for quiver varieties},
  Invent. Math. \textbf{135} (1999), no.~3, 665--687. \MR{1669280
  (2000f:14087)}

\bibitem[BKTY04]{abakhtay2004}
A.~S. Buch, A.~Kresch, H.~Tamvakis, and A.~Yong, \emph{Schubert polynomials and
  quiver formulas}, Duke Math. J. \textbf{122} (2004), no.~1, 125--143.
  \MR{2046809 (2005b:05218)}

\bibitem[BKTY05]{abakhtay2005}
\bysame, \emph{Grothendieck polynomials and quiver formulas}, Amer. J. Math.
  \textbf{127} (2005), no.~3, 551--567. \MR{2141644 (2007d:14018)}

\bibitem[BSY05]{abfsay2005}
A.~S. Buch, F.~Sottile, and A.~Yong, \emph{Quiver coefficients are {S}chubert
  structure constants}, Math. Res. Lett. \textbf{12} (2005), no.~4, 567--574.
  \MR{2155232 (2006g:14082)}

\bibitem[Buc02]{ab2002.qv}
A.~S. Buch, \emph{Grothendieck classes of quiver varieties}, Duke Math. J.
  \textbf{115} (2002), no.~1, 75--103. \MR{1932326 (2003m:14018)}

\bibitem[Buc05]{ab2005.alt}
\bysame, \emph{Alternating signs of quiver coefficients}, J. Amer. Math. Soc.
  \textbf{18} (2005), no.~1, 217--237 (electronic). \MR{2114821 (2006d:14052)}

\bibitem[Buc08]{ab2008}
\bysame, \emph{Quiver coefficients of {D}ynkin type}, Michigan Math. J.
  \textbf{57} (2008), 93--120, Special volume in honor of Melvin Hochster.
  \MR{2492443 (2009m:16032)}

\bibitem[BV82]{nbmv1982}
N.~Berline and M.~Vergne, \emph{Classes caract\'{e}ristiques \'{e}quivariantes.
  {F}ormule de localisation en cohomologie \'{e}quivariante}, C. R. Acad. Sci.
  Paris S\'{e}r. I Math. \textbf{295} (1982), no.~9, 539--541. \MR{685019}

\bibitem[BZ01]{gbgz2001}
G.~Bobi\'{n}ski and G.~Zwara, \emph{Normality of orbit closures for {D}ynkin
  quivers of type {$\Bbb A_n$}}, Manuscripta Math. \textbf{105} (2001), no.~1,
  103--109. \MR{1885816}

\bibitem[BZ02]{gbgz2002}
\bysame, \emph{Schubert varieties and representations of {D}ynkin quivers},
  Colloq. Math. \textbf{94} (2002), no.~2, 285--309. \MR{1967381}

\bibitem[Che14]{zc2014}
Zongbin Chen, \emph{Geometric construction of generators of {C}o{HA} of doubled
  quiver}, C. R. Math. Acad. Sci. Paris \textbf{352} (2014), no.~12,
  1039--1044. \MR{3276816}

\bibitem[Dav17]{bd2017}
B.~Davison, \emph{The critical {C}o{HA} of a quiver with potential}, Q. J.
  Math. \textbf{68} (2017), no.~2, 635--703. \MR{3667216}

\bibitem[DM16]{bdsm2016}
B.~Davison and S.~Meinhardt, \emph{Cohomological {D}onaldson-{T}homas theory of
  a quiver with potential and quantum enveloping algebras}, preprint, 2016.

\bibitem[DW17]{hdjw2017}
H.~Derksen and J.~Weyman, \emph{An introduction to quiver representations},
  Graduate Studies in Mathematics, vol. 184, American Mathematical Society,
  Providence, RI, 2017. \MR{3727119}

\bibitem[Efi12]{ae2012}
A.~I. Efimov, \emph{Cohomological {H}all algebra of a symmetric quiver},
  Compos. Math. \textbf{148} (2012), no.~4, 1133--1146. \MR{2956038}

\bibitem[FK94]{lfrk1994}
L.~D. Faddeev and R.~M. Kashaev, \emph{Quantum dilogarithm}, Modern Phys. Lett.
  A \textbf{9} (1994), no.~5, 427--434. \MR{1264393}

\bibitem[FR02]{lfrr2002.duke}
L.~Feh{\'e}r and R.~Rim{\'a}nyi, \emph{Classes of degeneracy loci for quivers:
  the {T}hom polynomial point of view}, Duke Math. J. \textbf{114} (2002),
  no.~2, 193--213. \MR{1920187 (2003j:14005)}

\bibitem[FR18]{hfmr2018}
H.~Franzen and M.~Reineke, \emph{Semistable {C}how-{H}all algebras of quivers
  and quantized {D}onaldson-{T}homas invariants}, Algebra Number Theory
  \textbf{12} (2018), no.~5, 1001--1025. \MR{3840869}

\bibitem[FR19]{hfmr2019}
\bysame, \emph{On the cohomological {H}all algebra of the {K}ronecker quiver},
  preprint, 2019.

\bibitem[Fra16]{hf2016}
H.~Franzen, \emph{On cohomology rings of non-commutative {H}ilbert schemes and
  {C}o{H}a-modules}, Math. Res. Lett. \textbf{23} (2016), no.~3, 805--840.
  \MR{3533197}

\bibitem[Fra18]{hf2018}
\bysame, \emph{On the semi-stable {C}o{H}a and its modules arising from smooth
  models}, J. Algebra \textbf{503} (2018), 121--145. \MR{3779991}

\bibitem[Gab72]{pg1972}
P.~Gabriel, \emph{Unzerlegbare {D}arstellungen. {I}}, Manuscripta Math.
  \textbf{6} (1972), 71--103; correction, ibid. 6 (1972), 309. \MR{0332887 (48
  \#11212)}

\bibitem[KKR19]{rkakjr2019}
R.~Kinser, A.~Knutson, and J.~Rajchgot, \emph{Three combinatorial formulas for
  type {$A$} quiver polynomials and {$K$}-polynomials}, Duke Math. J.
  \textbf{168} (2019), no.~4, 505--551. \MR{3916063}

\bibitem[KMS06]{akemms2006}
A.~Knutson, E.~Miller, and M.~Shimozono, \emph{Four positive formulae for type
  {$A$} quiver polynomials}, Invent. Math. \textbf{166} (2006), no.~2,
  229--325. \MR{2249801 (2007k:14098)}

\bibitem[KS11]{mkys2011}
M.~Kontsevich and Y.~Soibelman, \emph{Cohomological {H}all algebra, exponential
  {H}odge structures and motivic {D}onaldson-{T}homas invariants}, Commun.
  Number Theory Phys. \textbf{5} (2011), no.~2, 231--352. \MR{2851153
  (2012k:14079)}

\bibitem[Mil05]{em2005}
E.~Miller, \emph{Alternating formulas for {$K$}-theoretic quiver polynomials},
  Duke Math. J. \textbf{128} (2005), no.~1, 1--17. \MR{2137947 (2006e:05181)}

\bibitem[Rei03]{mr2003}
M.~Reineke, \emph{Quivers, desingularizations and canonical bases}, Studies in
  memory of {I}ssai {S}chur ({C}hevaleret/{R}ehovot, 2000), Progr. Math., vol.
  210, Birkh\"auser Boston, Boston, MA, 2003, pp.~325--344. \MR{1985731}

\bibitem[Rei10]{mr2010}
\bysame, \emph{Poisson automorphisms and quiver moduli}, J. Inst. Math. Jussieu
  \textbf{9} (2010), no.~3, 653--667. \MR{2650811}

\bibitem[Rim13]{rr2013}
R.~Rim\'{a}nyi, \emph{On the cohomological {H}all algebra of {D}ynkin quivers},
  preprint, 2013.

\bibitem[Rim14]{rr2014}
R.~Rim{\'a}nyi, \emph{Quiver polynomials in iterated residue form}, J.
  Algebraic Combin. \textbf{40} (2014), no.~2, 527--542. \MR{3239295}

\bibitem[Rim18]{rr2018}
R.~Rim\'{a}nyi, \emph{Motivic characteristic classes in cohomological {H}all
  algebras}, preprint, 2018.

\bibitem[Sut15]{ks2015}
K.~Sutar, \emph{Orbit closures of source-sink {D}ynkin quivers}, Int. Math.
  Res. Not. IMRN \textbf{~} (2015), no.~11, 3423--3444. \MR{3373055}

\end{thebibliography}
\end{document}